\documentclass{amsart}
\usepackage[margin=3cm]{geometry}
\usepackage{amsmath, amsthm, amssymb}
\usepackage{MnSymbol}
\newtheorem{theorem}{Theorem}[section]

\newtheorem{conjecture}[theorem]{Conjecture}
\newtheorem{proposition}[theorem]{Proposition}
\theoremstyle{definition}
\newtheorem{definition}[theorem]{Definition}
\newtheorem{example}[theorem]{Example}

\usepackage{hyperref}
\usepackage[usenames]{color}
\usepackage{epsfig}
\usepackage{graphicx}
\usepackage{caption}
\usepackage{subfig}
\usepackage{enumerate}
\theoremstyle{remark}
\newtheorem{remark}[theorem]{Remark}

\numberwithin{equation}{section}

\begin{document}

\title[Gradient Descent for finding Riemannian center of mass]{On The Convergence of Gradient Descent for Finding the Riemannian Center of Mass}
\author{Bijan Afsari}
\address{Center for Imaging Science, The Johns Hopkins University}
\curraddr{}
\email{bijan@cis.jhu.edu}
\thanks{The authors were supported in part by grants NSF CAREER \# 0447739,
NSF CNS \# 0834470, ONR \# N00014-05-10836, and ONR \# N00014-09-1-0084}
\author{Roberto Tron}
\address{Center for Imaging Science, The Johns Hopkins University}
\email{tron@cis.jhu.edu}
\author{Ren\'{e} Vidal}
\address{Center for Imaging Science, The Johns Hopkins University}
\email{rvidal@jhu.edu}



\maketitle

\begin{abstract}
We study the problem of finding the global Riemannian center of mass of a set of data points on a Riemannian manifold. Specifically, we investigate the convergence of constant step-size gradient descent algorithms for solving this problem. The challenge is that often the underlying cost function is neither globally differentiable nor convex, and despite this one would like to have guaranteed convergence to the global minimizer.  After some necessary preparations we state a conjecture which we argue is the best convergence condition (in a specific described sense) that one can hope for. The conjecture specifies conditions on the spread of the data points, step-size range, and the location of the initial condition (i.e., the region of convergence) of the algorithm. These conditions depend on the topology and the curvature of the manifold and can be conveniently described in terms of the injectivity radius and the sectional curvatures of the manifold.  For manifolds of constant nonnegative curvature (e.g., the sphere and the rotation group in $\mathbb{R}^{3}$) we show that the conjecture holds true (we do this by proving and using a comparison theorem which seems to be of a different nature from the standard comparison theorems in Riemannian geometry). For manifolds of arbitrary curvature we prove convergence results which are weaker than the conjectured one (but still superior over the available results).  We also briefly study the effect of the configuration of the data points on the speed of convergence.
\end{abstract}


\section{Introduction and Outline}
The (global) Riemannian center of mass (a.k.a. Riemannian mean or average or Fr\'{e}chet mean)\footnote{Throughout this paper we are mainly interested in the ``global'' Riemannian center of mass, hence in reference to it very often we drop the term ``global'' and simply use ``Riemannian center of mass,'' etc.  If need arises we explicitly use the term ``local'' in reference to a center which is not global (see Definition \ref{def:center_of_mass}).} of a set data points $\{x_{i}\}_{i=1}^{N}$ in a Riemannian manifold $M$ is defined as a point $\bar{x}\in M$ which minimizes the sum of squares of geodesic distances to the data points. This notion and its variants have a long history with several applications in pure mathematics (see e.g., \cite{Berger}, \cite{Grove3}, \cite{Karcher}, \cite{Kendall} and also \cite{BijanLp} for a brief history and some new results). More recently, statistical analysis on Riemannian manifolds and, in particular, the Riemannian center of mass have found applications in many applied fields. These include fields such as computer vision (see e.g., \cite{Meer} and \cite{Chellappa1}), statistical analysis of shapes (see e.g., \cite{Kendall_Book}, \cite{Le1}, \cite{Bhattacharaya1}, \cite{Groisser}, and \cite{Abishek}), medical imaging (see e.g., \cite{Goh_Vidal2009} and \cite{Pennec1}), and sensor networks (see e.g., \cite{Roberto1} and \cite{Sepulchre}), and many other general data analysis applications (see e.g., \cite{PhDKrakowski}, \cite{BijanPhd}, \cite{Buss}, and \cite{Moakher1}). In these applied settings one often needs to numerically locate or compute the Riemannian center of mass of a set of data points lying on a Riemannian manifold, and the so-called (intrinsic) gradient descent algorithm has been a popular choice for this purpose. The \emph{constant step-size} version of the gradient descent for finding the Riemannian center of mass is the easiest one to implement, and hence is the most popular one.

The main challenge in analyzing as well as working with the constant step-size gradient descent algorithm for finding the Riemannian center of mass is the fact that the underlying cost function is usually neither globally differentiable \footnote{For us global differentiability means differentiability everywhere on the manifold; however, we use the term ``global'' to remind ourselves that our functions of interest (e.g., the Riemannian distance from a point) lose differentiability at far away distances.} nor globally convex on the manifold.\footnote{In fact, it is well known that on a compact Riemannian manifold the only globally continuous convex functions are constant functions (see Theorem \ref{theorem:Yau} and \cite{ConvexFunction_Yau}).} More specifically, the cost function is not differentiable at any cut point of the data points. Moreover, as it can be shown by simple examples, the cost function can have local minimizers, which are not of interest and should be avoided. Nevertheless, we expect and hope that if the algorithm is initialized \emph{close enough} to the (unknown) global Riemannian center of mass $\bar{x}$ and the step-size is \emph{small enough}, then the algorithm should converge to the center. One would like to have the step-size small enough such that the cost is reduced at each step and at the same time the iterates do not leave a neighborhood around $\bar{x}$ in which $\bar{x}$ is the only zero of the gradient of the cost function (recall that a gradient descent algorithm, at best, can only locate a zero of the gradient vector field of the cost function). On the other hand, one would like to have large enough step-size so that the convergence is fast. The interplay between these three constraints is important in determining the conditions guaranteeing convergence as well as the speed of convergence. 

The goal of this paper is to give accurate conditions that guarantee convergence of the constant step-size gradient algorithm to the global Riemannian center of mass of a set of data points. In Section \ref{sec:preliminaries}, we first briefly give the necessary backgrounds on the Riemannian center of mass and the gradient descent algorithm for finding it, these include the notions of convex functions and sets in Subsection \ref{sec:convexity}, differentiability and convexity properties of the Riemannian distance function and bounds on its Hessian in Subsection \ref{sec:Hessian}, Riemannian center of mass and its properties in Subsection \ref{sec:Lp_center}, a general convergence theorem for gradient descent in Subsection \ref{sec:gradient_descent}, and a convergence theorem estimating the speed of convergence and the best step-size in Subsection \ref{sec:speed_of_convergence}. Following that, in Subsection \ref{sec:Conjecture}, we state Conjecture \ref{conj:optimal} in which we specify the \emph{best} ``condition for convergence'' one can hope for (in a sense to be described). Specifically, we specify a bound on the radius of any ball around the data points in which the algorithm can be initialized together with an interval of allowable step-size so that the algorithm converges to the global center of mass $\bar{x}$. The significant point is that for convergence the radius of the ball does not need to be any smaller than what ensures existence and uniqueness of the center. Moreover, the step-size can be chosen equal to the best (in a specific described sense) step-size under the \emph{extra} assumption that the iterates stay in that ball; \footnote{This step-size, in general, depends on an upper bound on the sectional curvatures of the manifold and the radius of the mentioned ball. However, interestingly, for a manifold of non-negative curvature it is simply $1$ (see Conjecture \ref{conj:optimal} and Remark \ref{rem:step_size}).} and it is conjectured that, indeed, the iterates stay in the ball. \footnote{The main challenge in proving this conjecture is to prove that the iterates stay in the ball containing the data points.} Knowing the conjecture helps us to compare and evaluate the existing results mentioned in Subsection \ref{sec:prior} as well as the results derived in this paper. In Section \ref{sec:optmial} (Theorem \ref{theorem:sphere_L2}), we prove Conjecture \ref{conj:optimal} for the case of manifolds of \textit{constant nonnegative} curvature. In our developments in this section, we first prove comparison Theorem \ref{theorem:comparison_sphere} in Subsection \ref{sec:comparison}. This comparison result (which differs from standard comparison theorems in some aspects) most likely has been known among geometers, but we could find neither its statement or a proof for it in the literature. In Subsection \ref{sec:convex_comb} we make sense of an intuitive notion of Riemannian affine convex combination of a set of data points in a manifold of nonnegative constant curvature and we explore its relation to the convex hull of the data points. These prepare us to prove the main theorem of the section, Theorem \ref{theorem:sphere_L2}, in Subsection \ref{sec:converge_optimal}. Although limited in scope, this result covers some very important manifolds of practical interest: $\mathbb{S}^{n}$ the unit sphere in $\mathbb{R}^{n+1}$, $SO(3)$ the group of rotations in $\mathbb{R}^{3}$, and $\mathbb{RP}^{n}$ the real projective space in $\mathbb{R}^{n+1}$.   In Section \ref{sec:sub_optimal}, for manifolds of arbitrary (i.e., \textit{non-constant}) curvature, we derive two classes of sub-optimal convergence conditions (in comparison with Conjecture \ref{conj:optimal}): In Subsection \ref{sec:data_spread_compromize} we give a result (Theorem \ref{theorem:compromised_L2}) in which convergence is guaranteed at the expense of smaller spread of data points, whereas in Subsection \ref{sec:step_size_compromise} (Theorem \ref{theorem:compromise_step_size_L2}) the allowable step-size is compromised to guarantee convergence. Finally, in Section \ref{sec:RateOfConvergence}, we qualitatively identify data configurations for which the convergence is very fast or very slow. We conclude the paper in Section \ref{sec:conc} by giving some remarks about further research.

\section{Preliminaries, a conjecture, and prior work}\label{sec:preliminaries}
\subsection{Preliminaries on the Riemannian Center of Mass and the Gradient Descent Algorithm}
\subsubsection{Notations}\label{sec:notations}
Let $M$ be an $n$-dimensional complete\footnote{Our results are local in nature, but they hold in relatively large regions that are determined explicitly. Therefore, completeness of the manifold is not necessary and our results could be easily adapted to e.g., non-singular regions of a singular manifold. However, for this purpose the statements of our results could become rather cumbersome.} Riemannian manifold with distance function $d$.\footnote{In our definitions relating to Riemannian manifolds we mainly follow \cite{Sakai}.} We denote the tangent space at $x\in M$ by $T_{x}M$ and by $\langle\cdot,\cdot\rangle$ and $\|\cdot\|$ we mean the Riemannian structure and the corresponding norm, respectively (dependence on the base point is implicit and clear from the context). By a $C^{k}$ function in a subset of $M$ we mean a continuous function which is $k^{th}$ order continuously differentiable in the subset as commonly understood in differential geometry. For a function $f:M\rightarrow \mathbb{R}$, $\nabla f$ denotes its gradient vector field with respect to $\langle\cdot,\cdot\rangle$. We assume that the sectional curvatures of $M$ are bounded from above and below by $\Delta$ and $\delta$, respectively. The exponential map of $M$ at $x\in M$ is denoted by $\exp_{x}(\cdot):T_{x}\rightarrow M$ and its inverse (wherever defined) is denoted by $\exp_{x}^{-1}(\cdot)$. The injectivity radius of $M$ is denoted by $\textrm{inj}M>0$.  An open ball with center $o\in M$ and radius $\rho$ is denoted by $B(o,\rho)$ and its closure by $\bar{B}(o,\rho)$.

\subsubsection{Convex functions and sets in Riemannian manifolds}\label{sec:convexity}
Convexity plays an important rule in our developments and we have the following definition:
\begin{definition}\label{def:convex_fct}
Let $A$ be an open subset of $M$ such that every two points in $A$ can be connected by at least one geodesic of $M$ such that this geodesic lies entirely in $A$. Assume that $f:A\rightarrow \mathbb{R}$ is a continuous function. Then $f$ is called (strictly) convex if the composition $f\circ\gamma:[0,1]\rightarrow \mathbb{R}$ is (strictly) convex for any geodesic $\gamma:[0,1]\rightarrow A$. We say that $f$ is globally (strictly) convex if it is (strictly) convex in $M$.
\end{definition}
If $f$ is $C^{2}$ in $A$, then convexity (strict convexity) of $f$ is equivalent to  $\frac{\texttt{d}^{2}}{\texttt{d}t^{2}}f(\gamma(t))|_{t=0}\geq 0$ ($>0$), where $\gamma:[0,1]\rightarrow A$ is any geodesic in $A$.

An insightful fact is the following \cite{ConvexFunction_Yau}:
\begin{theorem}\label{theorem:Yau}
The only globally convex function in a compact Riemannian manifold is a constant function.
\end{theorem}
It is more convenient to limit ourselves to \textit{strongly convex} subsets of $M$ because they are quite similar to standard convex sets in $\mathbb{R}^{n}$:
\begin{definition}
A set $A\subset M$ is called \textit{strongly convex} if any two points in $A$ can be connected by a unique minimizing geodesic in $M$ and the geodesic segment lies entirely in $A$.
\end{definition}
Define
\begin{equation}\label{eq:rcx}
r_{\textrm{cx}}=\frac{1}{2}\min\{\textrm{inj}M,\frac{\pi}{\sqrt{\Delta}}\},
\end{equation}
with the convention that $\frac{1}{\sqrt{\Delta}}=\infty$ for $\Delta\leq 0$. An open ball $B(o,\rho)$ with $\rho\leq r_{\textrm{cx}}$ is strongly convex in $M$; the same holds for any closed ball $\bar{B}(o,\rho)$ if $\rho<r_{\textrm{cx}}$ (see e.g., \cite[p. 404]{Sakai} and \cite[pp. 168-9]{Chavel}). In fact, $B(o,\rho)$ with $\rho\leq r_{\textrm{cx}}$ is even more similar to a convex set in Euclidean space: for any $x,y\in B(o,\rho)$ the minimal geodesic between from $x$ to $y$ is the only geodesic connecting them which lies entirely in $B(o,\rho)$.

\subsubsection{Differentiability and convexity properties of the distance function and estimates on its Hessian}\label{sec:Hessian}
Now, we briefly give some facts about the Riemannian distance function which will be used throughout the paper. Let $y\in M$. The function $x\mapsto d(x,y)$ is continuous for every $x\in M$. However, it is differentiable (in fact $C^{\infty}$) in $M\setminus (\{y\}\cup \mathcal{C}_{y})$, where $\mathcal{C}_{y}$ is the cut locus of $y$ (see e.g., \cite[pp. 108-110]{Sakai}). We recall the notion of cut locus: Let $\tilde{D}_{y}\in T_{x}M$ be the largest domain containing the origin of $T_{y}M$ in which $\exp_{y}:T_{y}\rightarrow M$ is a diffeomorphism, and let $\tilde{C}_{y}$ be the boundary of $\tilde{D}_{y}$. Then $\mathcal{C}_{y}=\exp_{y}(\tilde{C}_{y})$ is called the \textit{cut locus} of $y$ and one has $M=\exp_{y}(\tilde{D}_{y}\cup \tilde{C}_{y})$ (see e.g., \cite[pp. 117-118]{Chavel} or \cite[p. 104]{Sakai}). The distance between $y$ and $\mathcal{C}_{y}$ is called the injectivity radius of $y$ denoted by $\textrm{inj}y$, and by definition $\textrm{inj}M=\inf_{y\in M}\textrm{inj}y$. It is well known that $\mathcal{C}_{y}$ has measure zero in $M$. On the unit sphere $\mathbb{S}^{n}$ the cut locus of every point is its antipode. In a general manifold $M$ the differentiability property of $x\mapsto d(x,y)$ at $x=y$ is similar to the case where $M=\mathbb{R}^{n}$ equipped with the standard Euclidean metric; in particular, $x\mapsto \frac{1}{2}d^{2}(x,y)$ is a $C^{\infty}$ function in $M\setminus\mathcal{C}_{y}$. However, the behavior at far away points (e.g., the cut locus) is of substantially different nature and depends on the topology and curvature of $M$ (recall that in Euclidean space the cut locus of every point is empty).

Next, we recall some useful estimates about the Hessian of the Riemannian distance function. We adopt the following definitions:
\begin{equation}\label{eq:cot_hessian}
\begin{array}{cc}
\textrm{sn}_{\kappa}(l)=\left\{
\begin{array}{ll}
\frac{1}{\sqrt{\kappa}}\sin(\sqrt{\kappa} l)&\kappa > 0\\
\frac{1}{l}&\kappa =0\\
\frac{1}{\sqrt{|\kappa|}}\sinh(\sqrt{|\kappa|} l) &\kappa<0
\end{array}\right.
&
\textrm{ct}_{\kappa}(l)=\left\{
\begin{array}{ll}
\sqrt{\kappa}\cot(\sqrt{\kappa} l)&\kappa > 0\\
\frac{1}{l}&\kappa =0\\
\sqrt{|\kappa|}\coth(\sqrt{|\kappa|} l) &\kappa<0,
\end{array}\right.
\end{array}
\end{equation}
and
\begin{equation}\label{eq:c_kappa}
\begin{array}{cc}
\textrm{b}_{\kappa}(l)=\left\{
\begin{array}{ll}
\sqrt{\kappa}l\cot(\sqrt{\kappa} l) &\kappa \geq 0\\
1 &\kappa<0
\end{array}\right.
&
\textrm{c}_{\kappa}(l)=\left\{
\begin{array}{ll}
1 &\kappa \geq 0\\
\sqrt{|\kappa|}l\coth(\sqrt{|\kappa|} l) &\kappa<0.
\end{array}\right.
\end{array}
\end{equation}

Assume that $x\in M$ is such that $d(x,y)<\min\{\textrm{inj}y,\frac{\pi}{\sqrt{\Delta}}\}$.\footnote{Instead of this condition, it is often more convenient to require $d(x,y)<\min\{\textrm{inj}M,\frac{\pi}{\sqrt{\Delta}}\}$, which is a more conservative yet global version of the condition.} Furthermore, assume that $t\mapsto \gamma(t)$ with $\gamma(0)=x$ is a unit speed geodesic making, at $x$, an angle $\beta$ with the \textit{minimal} geodesic from $y$ to $x$. It can be proved that (see e.g., \cite[pp. 152-154]{Sakai})
\begin{equation}\label{eq:hessian_d_bounds}
\textrm{ct}_{\Delta}(d(x,y))\sin^{2}\beta\leq \frac{\mathrm{d}^{2}}{\mathrm{d}t^{2}}d(\gamma(t),y)\big|_{t=0}\leq \textrm{ct}_{\delta}(d(x,y))\sin^{2}\beta,
\end{equation}
where $\textrm{ct}_{\kappa}$ is defined in (\ref{eq:cot_hessian}). Based on the above one can verify that
\begin{equation}\label{eq:hessian_f2_bounds}
\textrm{b}_{\Delta}(d(x,y))\leq \frac{\mathrm{d}^{2}}{\mathrm{d}t^{2}}\big(\frac{1}{2}d^{2}(\gamma(t),y)\big)\big|_{t=0}\leq \textrm{c}_{\delta}(d(x,y)),
\end{equation}
and more generally that
\begin{equation}\label{eq:hessian_fp_bounds}
d^{p-2}(x,y)\min\{p-1,\textrm{b}_{\Delta}(d(x,y))\}\leq \frac{\mathrm{d}^{2}}{\mathrm{d}t^{2}}\big(\frac{1}{p}d^{p}(\gamma(t),y)\big)\big|_{t=0}\leq d^{p-2}(x,y)\max\{p-1,\textrm{c}_{\delta}(d(x,y))\}.
\end{equation}
Notice that the above confirms the intuition that the differentiability behavior of $x\mapsto \frac{1}{p}d^{p}(x,y)$ at $y=x$ is the same in $M$ and in $\mathbb{R}^{n}$. We will use these two relations very often; in doing so it is useful to have in mind that $\textrm{b}_{\kappa}(l)$ and $\textrm{c}_{\kappa}(l)$, respectively, are decreasing and increasing in $l$.
\begin{remark}
We emphasize that the requirement $d(x,y)<\textrm{inj}y$ is essential and cannot be compromised, as at a cut point the distance function becomes non-differentiable. Although in the left hand side of the above bounds only $\Delta$ appears explicitly both the curvature and topology of $M$ determine the convexity properties of the distance function. Notice that $\Delta$ gives (some) information about the Riemannian curvature tensor of $M$ and $\textrm{inj}y$ (or $\textrm{inj}M$) gives (some) information about the topology of $M$. For example, $\mathbb{R}$ and the unit circle $\mathbb{S}^{1}$ both have zero sectional curvature, while $\textrm{inj}\mathbb{S}^{1}=\pi$ and the injectivity radius of $\mathbb{R}$ is infinity. Now obviously $x\mapsto \frac{1}{2}d^{2}(x,y)$ in $\mathbb{R}$ is globally convex, while in $\mathbb{S}^{1}$ it is not (as seen directly or as a consequence of Theorem \ref{theorem:Yau}). The interesting fact is that $x\mapsto \frac{1}{2}d^{2}(x,y)$  has positive definite Hessian in $\mathbb{S}^{1}\setminus \{y'\}$, where $y'$ is the antipodal point of $y$. At $y'$, $x\mapsto \frac{1}{2}d^{2}(x,y)$ is not differentiable. This non-differentiability is so severe that in spite of the positivity of its Hessian at all other points $x\mapsto \frac{1}{2}d^{2}(x,y)$ is far from a globally convex function on $\mathbb{S}^{1}$. As an example of a less severe non-differentiability notice that $x\mapsto d(x,y)$ in $\mathbb{R}$ is also non-differentiable at $x=y$, yet this does not affect its convexity.
\end{remark}

\subsubsection{Riemannian $L^{p}$ center of mass}\label{sec:Lp_center}
We start by the following definition:
\begin{definition}\label{def:center_of_mass}
The (global) the Riemannian $L^{p}$ center of mass or mean (a.k.a. Fr\'{e}chet mean) of the data set $\{x_{i}\}_{i=1}^{N}\subset M$ with respect to weights $0\leq w_{i}\leq 1$ ($\sum_{i=1}^{N}w_{i}=1$) is defined as the minimizer(s) of
\begin{equation}\label{eq:fp}
f_{p}(x)=\left\{\begin{array}{lc}
\frac{1}{p}\sum_{i=1}^{N} w_{i} d^{p}(x,x_{i}) &1\leq p<\infty\\
\max_{i}d(x,x_{i})& p=\infty,\\
\end{array}
\right.
\end{equation}
in $M$. We denote the center by $\bar{x}_{p}$. We call a local minimizer of $f_{p}$ (which is not global) a local center of mass of the data set with respect to the weights.\footnote{A local minimizer of $f_{p}$ in $M$ is sometimes called a Karcher mean, although this definition bears little relation with the way Grove and Karcher \cite{Grove3} or Karcher \cite{Karcher} originally defined what they called the Riemannian center of mass (see also \cite{BijanLp} for more details). Given $\{x_{i}\}_{i=1}^{N}\subset B(o,\rho)$ with small enough $\rho$ those authors defined the center of mass as a zero of $\nabla f_{2}$ in $\bar{B}(o,\rho)$ or alternatively as a local minimizer of $f_{2}$ in $\bar{B}(o,\rho)$ (Notice that given that $f_{2}$ is not differentiable at the cut locus of each data point it is not a-priori clear that, in general, a local minimizer of $f_{2}$ in $M$ should coincide with a zero of $\nabla f_{2}$, although there are some evidence in the literature that this situation might not happen, see e.g,. \cite{Circle_Huckemann}, \cite{Circle_Charlier}). Overall, there is no consensus among authors about the terminology and we find it more convenient to use ``local'' and ``global'' Riemannian center of mass as defined here.}
\end{definition}
The reader is referred to \cite{BijanLp} for details and other related definitions. As a convention when referring to the center of mass of some data points we usually do not refer to explicit weights unless needed. As another convention when $p$ is not specified we assume $p=2$, which is the most commonly used case.  Although $p=1$ and $p=\infty$ are also used often, in this paper our focus is limited to $2\leq p<\infty$.\footnote{Going from $p=2$ to $2<p<\infty$ does not introduce any major technical challenge, so our analysis can be applied almost uniformly to $2\leq p<\infty$. However, since the results for $p=2$ are easier to state and presumably used more often, we usually state a result for $p=2$ and then give the more general version for $2\leq p<\infty$. The only exception is Theorem \ref{theorem:compromise_step_size_L2} for which we do not give a $2<p<\infty$ version.} The reason is that in our analysis we require $f_{p}$ to be twice-continuously differentiable and we determine the constant step-size of the gradient algorithm in terms of the upper bounds on the eigenvalues Hessian of $f_{p}$. As in Euclidean case, in the more general Riemannian case also one can see from (\ref{eq:hessian_fp_bounds}) that for $1\leq p<2$ the Hessian of $f_{p}$ might be unbounded. It is well known that Lipschitz continuous Hessian (in particular bounded Hessian) is necessary for the convergence of a gradient descent method with constant-step size \cite{Polyak}.

Although $f_{p}:M\rightarrow \mathbb{R}$ is a globally convex function when $M$ is a Euclidean space (or more generally an Hadamard manifold), $f_{2}$ is not globally convex for an arbitrary $M$. In particular, the center of mass might not be unique; however, if the data points are close enough, then the center is unique. The following theorem gives sufficient conditions for existence and uniqueness of the Riemannian center of mass.

\begin{theorem}\label{theorem:exist_unique}
Let $\{x_{i}\}_{i=1}^{N}\subset B(o,\rho)$ and assume $0\leq w_{i}\leq 1$ with $\sum_{i=1}^{N}w_{i}=1$.  For $p\geq 2$, if $\rho\leq r_{\textrm{cx}}$, then the Riemannian $L^{p}$ center of mass $\bar{x}_{p}$ is unique and is inside $B(o,\rho)$. For $2\leq p<\infty$ it is the unique zero of the gradient vector field $\nabla f_{p}$ in $\bar{B}(o,\rho)$, and moreover if no data point has weight $1$, then $\bar{x}_{p}$ is a non-degenerate critical point of $f_{p}$ (i.e., the Hessian of $f_{p}$ at $\bar{x}_{p}$ is positive-definite). \footnote{In Theorem 2.1 of \cite{BijanLp}, the condition on $\rho$ is stated as $\rho<r_{\textrm{cx}}$, but since we have finite number of data points the current version follows immediately. Also from the statement of Theorem 2.1, $\bar{x}_{p}$ is the only zero of $\nabla f_{p}$ in $B(o,\rho)$, but from the proof of the theorem it can be seen that the vector field $-\nabla f_{p}$ is inward-pointing on the boundary of $B(o,\rho)$ and hence $\bar{x}_{p}$ is the only zero in entire $\bar{B}(o,\rho)$. In addition, the fact about non-degeneracy is not present in the statement of Theorem 2.1 of \cite{BijanLp}, however it is proved in the proof of the theorem.}
\end{theorem}
For a proof see \cite{BijanLp}. Also for $1\leq p<2$ see \cite{BijanLp} and \cite{Yang1}. In this paper, we are mainly interested in $2\leq p<\infty$, because the bounds we derive on the step-size of the gradient descent algorithm depend on an upper bound on the eigenvalues of the Hessian of $f_{p}$, and for $1\leq p<2$ (and $p=\infty$) the Hessian is unbounded. We refer the reader to \cite{Yang1} and \cite{Arnaudon_L_infty} for algorithms in the case of $p=1$ and $p=\infty$.

A rather subtle issue is that according to Theorem \ref{theorem:exist_unique}, if $\rho\leq r_{\textrm{cx}}$, then $f_{p}$ has a unique minimizer in $\bar{B}(o,\rho)$. Nevertheless, $f_{p}|_{B(o,\rho)}$, the restriction of $f_{p}$ to $B(o,\rho)$, might not be a convex function. More accurately, if $\rho\leq \frac{1}{2}\min\{\textrm{inj}M,\frac{\pi}{2\sqrt{\Delta}}\}$ (cf. (\ref{eq:rcx})), then $f_{p}|_{B(o,\rho)}$ is a convex function (this can be seen easily from (\ref{eq:hessian_f2_bounds}) by noting that $b_{\Delta}(x)>0$ for $x\in B(o,\rho)$). However, if $\Delta>0$ and $\frac{\pi}{4\sqrt{\Delta}}<\rho\leq r_{\textrm{cx}}$, then $f_{p}|_{B(o,\rho)}$ might not be a convex function, as can be seen by very simple examples. While this fact does not harm implementation and convergence of a gradient descent algorithm for finding $\bar{x}_{p}$ greatly, it has some serious implications on implementation and applicability of Newton's method (see Remarks \ref{rem:gradient_convexity} and \ref{rem:Grad_vs_Newton}).

\subsubsection{Gradient descent algorithm for finding the Riemannian center of mass}\label{sec:gradient_descent}
For later reference we derive the intrinsic gradient descent algorithm for locating the Riemannian $L^{p}$ center of mass (see \cite{AbsMahSep2008} or \cite{Udriste_Book} for introduction to optimization on Riemannian manifolds). One can check that
\begin{equation}\label{eq:grad_Lp}
\nabla f_{p}(x)=-\sum_{i=1}^{N} w_{i}d^{p-2}(x,x_{i})\exp_{x}^{-1}x_{i},
\end{equation}
for any $x\in M$ as long as it is not in the cut locus of any of the data points. In particular, if $\{x_{i}\}_{i=1}^{N}\subset B(o,\rho)$, where $\rho<\frac{\textrm{inj}M}{2}$, then for any $x\in B(o,\rho)$ the above expression is well-defined in the classical sense (i.e., it is uniquely defined). Notice  that the above expression is well-defined for almost every $x\in M$ because the set at which $f_{p}$ is not differentiable has measure zero (for $p>1$ this set is $\cup_{i} \mathcal{C}_{x_{i}}$ and for $p=1$ it is $\cup_{i} \mathcal{C}_{x_{i}}\cup \{x_{i}\}_{i}$). As we will see this non-differentiability has severe implications on the behavior of the constant step-size gradient descent.

Algorithm \ref{table:T1} in Table \ref{table:T1} is a gradient descent algorithm for locating the Riemannian center of mass of $\{x_{i}\}_{i=1}^{N}$.
\begin{table}[h]
  \centering
  \begin{tabular}{|l|}
  \hline
    \textbf{Algorithm 1: Gradient Descent for finding the Riemannian $L^{p}$ center of mass}\\
    \parbox[l]{10cm}{
 \begin{enumerate}
  \item Consider $\{x_{i}\}_{i=1}^{N}\subset B(o,\rho) \subset M$ and weights $\{w_{i}\}_{i=1}^{N}$ and choose $x^{0}\in M$.
  \item\label{item:step2} \texttt{if} $\nabla f_{p}(x^{k})=0$ \texttt{then} stop, \texttt{else set}
  \begin{equation}\label{eq:gradient_descent}
  x^{k+1}=\exp_{x^{k}}(-t_{k}\nabla f_{p}(x^{k}))
  \end{equation}
  where $t_{k}>0$ is an ``appropriate'' step-size and $\nabla f_{p}(\cdot)$ is defined in (\ref{eq:grad_Lp}).
  \item \texttt{goto} step \ref{item:step2}.

  \end{enumerate}
  }\\
  \hline
  \end{tabular}
  \caption{Gradient descent for finding the Riemannian $L^{p}$ center of mass.}\label{table:T1}
\end{table}
Besides practical considerations (e.g., stopping criterion), at least two important issues are left unspecified in Algorithm \ref{table:T1}, namely, how to choose $x^{0}$ and how to choose $t_{k}$ for each $k$. The most natural choice for $x^{0}$ is one point in $B(o,\rho)$, say one of the data points. Note that in practice $o$ and the exact value of $\rho$ might not be known.

The choice of $t_{k}$ is more complicated. The next general proposition gives a prescription for a step-size interval which ensures reducing the cost function at an iteration of a gradient descent provided one knows an upper bound on the eigenvalues of the Hessian of the cost function in a region which iterates live. The proof of this proposition follows from the second order Taylor series expansion (with remainder).

\begin{proposition}\label{prop:hessian_stepsize}
Let $x\in M$ and consider an open neighborhood $S\subset M$ containing $x$. Let $f:M\rightarrow \mathbb{R}$ be a function whose restriction to $S$ is twice-continuously differentiable and let the real number $H_{S}$ be an upper bound on the eigenvalues of the Hessian of $f$ in $S$. There exists $ t_{x,S}>0$ such that for all $t\in [0,t_{x,S})$ the curve $t\mapsto \exp_{x}(-t\nabla f(x))$ does not leave $S$ and
\begin{equation}\label{eq:hessian_reduction}
f(\exp_{x}(-t\nabla f(x)))\leq f(x)-\|\nabla f(x)\|^{2}t+ \frac{H_{S}\|\nabla f(x)\|^{2}}{2}t^{2}.
\end{equation}
For $t\in (0,\min\{t_{x,S},\frac{2}{H_{S}}\})$, with the convention that $\frac{1}{H_{S}}=+\infty$ for $H_{S}\leq 0$, we have $f(\exp_{x}(-t\nabla f(x)))\leq f(x)$ with equality only if $x$ is a critical point of $f$. Moreover, when $H_{S}>0$ the right hand side of (\ref{eq:hessian_reduction}) is minimized for $t=\frac{1}{H_{S}}$.
\end{proposition}

Notice that the fact that for $t\in [0,t_{x,S})$ the curve $t\mapsto \exp_{x}(-t\nabla f(x))$ stays in $S$ is crucial in enabling us to use the upper bound $H_{S}$ and derive (\ref{eq:hessian_reduction}). This concept appears frequently in our analysis and it useful to have the following definition:
\begin{definition}
Let $x^{k}\in S\subset M$. We say that iterate $x^{k+1}$ of Algorithm \ref{table:T1} \textit{stays} in $S$ if $x^{k+1}=\exp_{x^{k}}(-t_{k}\nabla f_{p}(x^{k}))\in S$. We say that the iterate $x^{k+1}$ of Algorithm \ref{table:T1} \textit{continuously stays} in $S$ if $\exp_{x^{k}}(-s\nabla f_{p}(x^{k}))\in S$ for $s\in [0,t_{k}]$.
\end{definition}
Obviously, continuously staying in $S$ is stronger than staying in $S$. However, they are equivalent under some conditions (which hold in some, but not all, of the cases we study, see Remark \ref{rem:cont_stay}):
\begin{proposition}\label{prop:cont_stay}
If $S$ is a strongly convex set and $t_{k}\|\nabla f_{p}(x^{k})\|<\textup{\textrm{inj}}M$ for every $k\geq 0$, then for the iterates of Algorithm \ref{table:T1} staying in $S$ implies (and hence is equivalent to) continuously staying in $S$.
\end{proposition}
\begin{proof}
Assume that $x^{k}$ and $x^{k+1}$ both belong to $S$. Recall that $t\mapsto \exp_{x}(-t\nabla f_{p}(x^{k}))$ for $t\in [0,t_{k}]$ is a minimizing geodesic if $t_{k}\|\nabla f_{p}(x^{k})\|<\textrm{inj}M$. Therefore, by strong convexity of $S$, $t\mapsto \exp_{x}(-t\nabla f_{p}(x^{k}))$ for $t\in [0,t_{k}]$ must be the only minimizing geodesic between $x^{k}$ and $x^{k+1}$ and must lie in $S$ entirely.
\end{proof}

As mentioned before we are only interested in constant step-size gradient descent. The following convergence result is standard for this type of algorithms when the cost is $C^{2}$ (or at least has Lipschitz gradient) and is globally convex; however, our version is adapted to $f_{p}$ which, in general, is neither globally $C^{2}$ nor convex. The assumption of the theorem that each iterate of the algorithm continuously stay in a neighborhood $S$ of $\bar{x}_{p}$, in which is $\bar{x}_{p}$  is the only zero of $\nabla f_{p}$, is a crucial enabling ingredient of the proof. In fact, our goal in Sections \ref{sec:optmial} and \ref{sec:sub_optimal} is essentially to identify such a neighborhood (under certain conditions).
\begin{theorem}\label{theorem:convergence_general}
Let $2\leq p<\infty$ and assume that $\bar{x}_{p}$ is the center of mass of $\{x_{i}\}_{i=1}^{N}\subset B(o,\rho)$, where $\rho\leq r_{\textrm{cx}}$. Let $S$ be a bounded open neighborhood of $\bar{x}_{p}$ such that $f_{p}$ is $C^{2}$ in $S$ and $C^{1}$ in $\bar{S}$, the closure of $S$. Furthermore, assume that $\bar{x}_{p}$ is the only zero of the vector field $\nabla f_{p}$ in $\bar{S}$. Let $H_{S}$ be an upper bound on the eigenvalues of the Hessian of $f_{p}$ in $S$. In Algorithm \ref{table:T1} choose $t_{k}=t\in(0,\frac{2}{H_{S}})$. If starting from $x^{0}\in S$, each iterate of Algorithm \ref{table:T1} continuously stays in $S$, then $f_{p}(x^{k+1})\leq f_{p}(x^{k})$ for $k\geq 0$ with equality only if $x^{k}=\bar{x}_{p}$ and $x^{k}$ converges to $\bar{x}_{p}$.
\end{theorem}
\begin{proof}
Since by assumption $x^{k}\in S$ for $k\geq 0$ and $\bar{S}$ is compact, there is a subsequence $\langle x^{k_{j}}\rangle_{k_{j}}$ converging to a point $x^{*}\in \bar{S}$. By Proposition \ref{prop:hessian_stepsize}  we have $f_{p}(x^{k+1})\leq f_{p}(x^{k})$ unless $x^{k}=\bar{x}_{p}$ and furthermore
\begin{equation}
t(1-\frac{H_{S}t}{2})\sum_{j=1}^{k}\|\nabla f_{p}(x^{j})\|^{2}\leq f_{p}(x^{k+1})-f_{p}(x^{0}),
\end{equation}
for every $k\geq 0$. Since $\langle f_{p}(x^{k})\rangle_{k}$ is a bounded sequence, the above implies that $\|\nabla f_{p}(x^{k})\|\rightarrow 0$; hence, by continuity of $\nabla f_{p}$ we have $\|\nabla f_{p}(x^{*})\|=0$, that is, $x^{*}$ is a zero of $\nabla f_{p}$ in $\bar{S}$. But by the assumption about $S$ this means that $x^{*}$ coincides with $\bar{x}_{p}$ and therefore $x^{k}\rightarrow \bar{x}_{p}$.
\end{proof}

Next, we give a very simple but insightful example.
\begin{example}[Finding the mean of two points on the unit circle $\mathbb{S}^{1}$]\label{ex:circle}
Let $M$ be the unit circle $\mathbb{S}^{1}$ centered at the origin $(0,0)$ and equipped with the standard arc length distance $d$. Recall that $\Delta=0$ and $\textrm{inj}M=\pi$.  Let $o$ denote the point $(1,0)$ (see Figure \ref{fig:circle_convergence}). We consider two data points $x_{1},x_{2}\in \mathbb{S}^{1}$ represented as $x_{i}=(\cos\theta_{i},\sin\theta_{i})$ ($i=1,2$) where $0<\theta_{1}<\rho\leq \frac{\pi}{2}$ and $\theta_{2}=-\theta_{1}$. We specify the weights and $\theta_{1}$ later. Under the mentioned assumption that $\rho\leq \frac{\pi}{2}$, Theorem \ref{theorem:exist_unique} guarantees that the center of mass $\bar{x}$ is unique and in fact one can see that $\bar{x}=(\cos\bar{\theta},\sin\bar{\theta})$ where $\bar{\theta}=w_{1}\theta_{1}+w_{2}\theta_{2}$. More importantly, it also follows that $\bar{x}$ is the unique zero of $\nabla f_{2}$ in $B(o,\rho)$ (as well as $B(o,\frac{\pi}{2})$). Notice that $f_{2}$ is smooth within $B(o,\pi-\theta_{1})$ (it does not follow from Theorem \ref{theorem:exist_unique} but it is an easily verifiable fact that $\bar{x}$ is the unique zero of $\nabla f_{2}$ in $B(o,\pi-\theta_{1})$). On the other hand, $f_{2}$ is not differentiable at $x_{1}'$ and $x_{2}'$, the antipodal points of $x_{1}$ and $x_{2}$, respectively. Furthermore, in $\mathbb{S}^{1}\setminus \{x_{1}',x_{2}'\}$  the Hessian of $f_{2}$ is defined and is equal to $1$, hence the largest possible range of the constant step-size $t_{k}=t$ is the interval $(0,2)$. Next, we see under what conditions Theorem \ref{theorem:convergence_general} applies.
It is easy to check that, independent of the weights, with $x^{0}\in B(o,\rho)$ and step-size $t_{k}=t\in (0,1]$ the iterates of Algorithm \ref{table:T1} continuously stay in $B(o,\rho)$. Therefore, an acceptable $S$ is the ball $B(o,\rho)$ and Theorem \ref{theorem:convergence_general} ensures convergence to the global center $\bar{x}$ provided step-size $t$ is in the interval $(0,1]$. This result is essentially not different from what we have in $\mathbb{R}$.
However, the situation for $t\in (1,2)$ is rather subtle since with a large step-size the iterates might leave the ball $B(o,\rho)$ or even $B(o,\pi-\theta_{1})$ and enter a region in which there is another zero of $\nabla f_{2}$. To be specific notice that $f_{2}$ can be parameterized with $\theta \in (-\pi,+\pi]$ as
\begin{equation}
f_{2}(\theta)=\frac{1}{2}\left\{
\begin{array}{ll}
w_{1}(\theta-\theta_{1}-2\pi)^{2}+w_{2}(\theta-\theta_{2})^{2}& -\pi< \theta\leq \theta_{1}-\pi \\
w_{1}(\theta-\theta_{1})^{2}+w_{2}(\theta-\theta_{2})^{2}& \theta_{1}-\pi\leq \theta\leq \theta_{2}+\pi \\
w_{1}(\theta-\theta_{1})^{2}+w_{2}(\theta-\theta_{2}+2\pi)^{2}& \theta_{2}+\pi\leq \theta\leq \pi. \\
\end{array}
\right.
\end{equation}
Now, let us fix $\theta_{1}=\frac{2\pi}{5}$ (and $\theta_{2}=-\frac{2\pi}{5}$). First, let $w_{1}=\frac{1}{10}$ and $w_{2}=\frac{9}{10}$. The solid curve in the right panel in Figure \ref{fig:circle_convergence} shows the graph of $f_{2}(\theta)$. The two cranks in the curve at $\theta_{1}'=\frac{-3\pi}{5}$ and $\theta_{2}'=\frac{3\pi}{5}$ are due to the non-differentiability of $f_{2}$ at antipodal points of $x_{1}$ and $x_{2}$. If we run Algorithm \ref{table:T1} with $x^{0}=x_{1}$ and step-size $t=\frac{25}{18}$ we have $x^{1}=x_{1}'$, thus $x^{1}$ coincides with a non-differentiable critical point of $f_{2}$ at which the algorithm is, in fact, not well-defined. In the generic setting the probability of this happening is zero; however, for larger $t$, $x^{1}$ will leave $B(o,\pi-\theta_{1})$. It can be seen that for this specific pair of weights $\nabla f_{2}$ has only one zero in $\mathbb{S}^{1}$. Consequently, in practice, Algorithm \ref{table:T1} for almost every initial condition in $\mathbb{S}^{1}$ and step-size $t_{k}=t$ in the interval $(0,2)$ will find the global center of mass $\bar{x}=(\cos\bar{\theta},\sin\bar{\theta})$, where $\bar{\theta}=\frac{-8\pi}{25}$ (this fact does not follow from Theorem \ref{theorem:convergence_general} but is not difficult to verify in this special case, see also Remark \ref{rem:local_vs_global}). But we might not be this lucky always! For example, let $w_{1}=\frac{1}{4}$ and $w_{2}=\frac{3}{4}$. The dashed curve in Figure \ref{fig:circle_convergence} show $f_{2}(\theta)$ for this pair of weights. One can verify that in addition to the global minimizer $\bar{\theta}=\frac{-\pi}{5}$, this time, $f_{2}(\theta)$ has a local minimizer at $\bar{\theta}'=\frac{-7\pi}{10}$. Now if we run Algorithm \ref{table:T1} with $x^{0}=x_{1}$ and constant step-size $t_{k}=t$ where $t=\frac{11}{6}$, then we have $x^{1}=\frac{-7\pi}{10}$, i.e., the next iterate coincides with the local center $\bar{x}'=(\cos\bar{\theta}',\sin\bar{\theta}')$ and the algorithm gets stuck at the wrong center! For values of $t$ slightly smaller or larger than  $\frac{11}{6}$ the algorithm still converges to $\bar{x}'$. Notice that this happens despite the fact that the cost is reduced at each iteration.\footnote{It would be interesting to see whether an example (in $\mathbb{S}^{1}$ or another manifold) exists in which due to the non-differentiability of $f_{2}$, we have $f_{2}(x^{1})>f_{2}(x^{0})$ if $x^{1}$ does not continuously stay in $S$. Such a phenomenon could lead to an oscillatory behavior (see Remark \ref{rem:local_vs_global}).} Although this simple example does not show the effect of the curvature of the manifold, still it makes it clear that in order for Algorithm \ref{table:T1} to have a predictable behavior that is as data-independent as possible it is important to identify conditions under which the assumptions of Theorem \ref{theorem:convergence_general} are satisfied (mainly that the iterates continuously stay in a candidate $S$). Our efforts in the following sections are primarily directed toward finding a set $S$, the radius $\rho$ of a ball containing the data points, and the step-size range which guarantee that the iterates continuously stay in $S$.

\end{example}

\begin{figure}[h]
  \begin{center}
  $
  \begin{array}{cc}
  \scalebox{.75}{\input{local_global_centers.pstex_t}} &
  \scalebox{.5}{\includegraphics{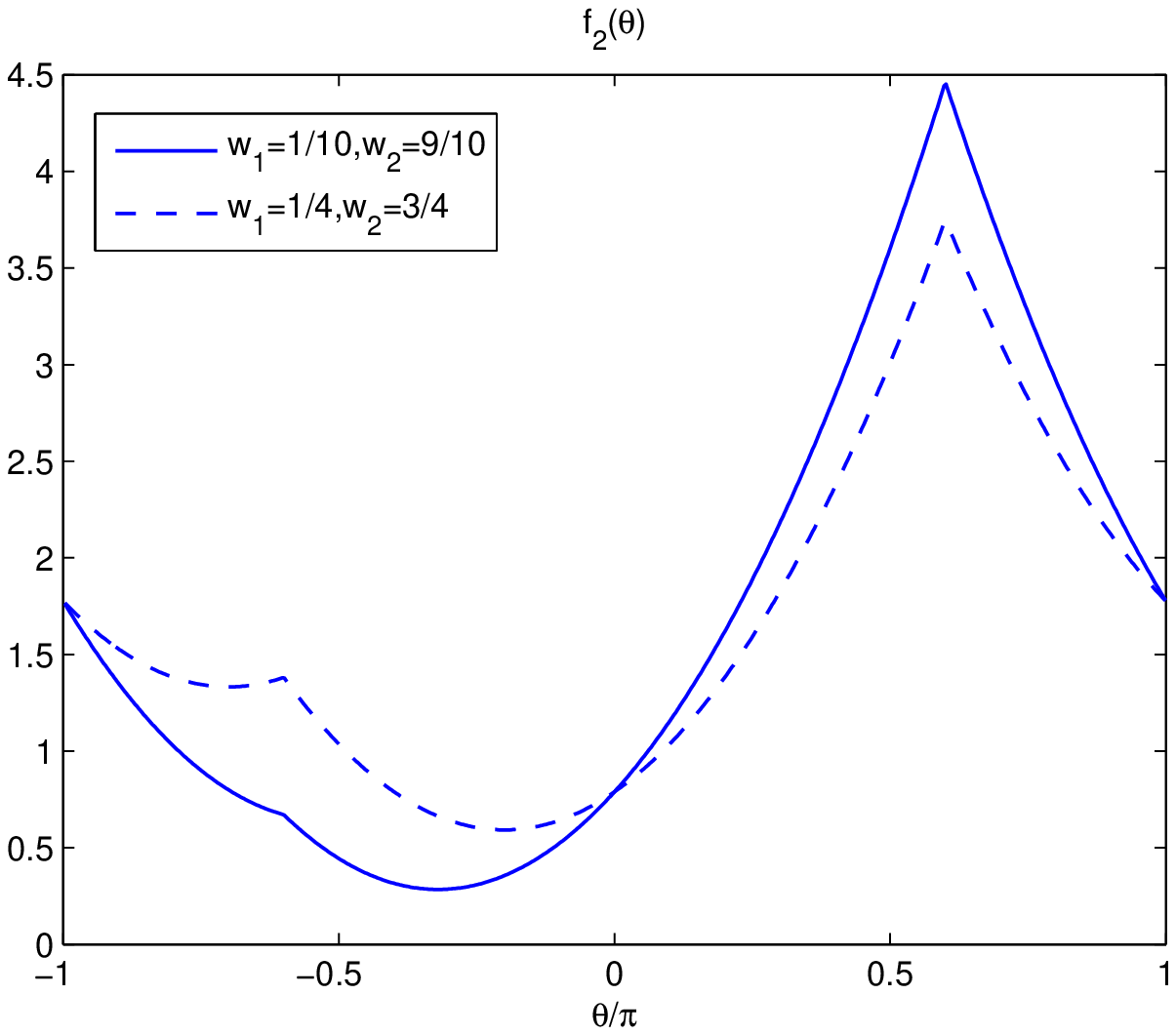}}\\
  \end{array}
  $
  \caption{The left panel shows the configuration of data points $x_{1}$ and $x_{2}$ in Example \ref{ex:circle} and the right panel is the graph of $f_{2}(\theta)$ for two different pairs of weights. The function $f_{2}$ is non-differentiable at $x_{1}'$ and $x_{2}'$, the antipodal points of $x_{1}$ and $x_{2}$. The example shows that if the step-size is large the iterates might leave the region in which $f_{2}$ is smooth and the algorithm might converge to $\bar{x}'$, a local center of $x_{1}$ and $x_{2}$, instead of the global center $\bar{x}$ (despite the fact that the cost is reduced at each step). Notice that the correct way of thinking about the plotted graphs it to visualize them while identifying points $\theta=-\pi$ and $\theta=+\pi$ or to think of them as periodic graphs with period $2\pi$.
  }\label{fig:circle_convergence}
\end{center}
\end{figure}
\begin{remark}[Local vs.global behavior]\label{rem:local_vs_global}
Our focus in this paper is on finding the global center of mass using a constant step-size gradient descent algorithm; hence, we only study the \textit{local behavior} of the algorithm (albeit in relatively large domains), and we leave the \textit{global behavior} analysis of the algorithm to further research. In this regard, one particular important question is: ``How (if possible at all) could one guarantee that for almost \textit{every} initial condition in $M$ the algorithm converges to \textit{a local} Riemannian center of mass?'' Notice that this is a relevant question, in particular, since the underling cost function is not differentiable globally and e.g., a-priori one could not rule out a possible oscillatory behavior.
\end{remark}

\subsubsection{Speed of convergence and the best step-size}\label{sec:speed_of_convergence}

In Proposition \ref{prop:hessian_stepsize}, $t=\frac{1}{H_{S}}$ is the best step-size in the sense that in each iteration it gives the most reduction in the upper bound of $f_{p}(x^{k+1})$ described in the right hand side of (\ref{eq:hessian_reduction}). The following theorem relates this choice to the speed of convergence of the algorithm. The proof of the theorem is adopted from \cite[p. 266, Theorem 4.2]{Udriste_Book} where a proof is given for a globally convex $C^{2}$ function. Here, we adapt that proof to constant step-size gradient descent for minimizing $f_{p}$ (which is only locally $C^{2}$ and convex).\footnote{It should be clear that nothing is special about $f_{p}$ and $\bar{x}_{p}$ in Theorems \ref{theorem:convergence_general} and \ref{theorem:rate_of_convergence}: Both theorems hold true if $f_{p}$ is replaced with any function $f:M\rightarrow \mathbb{R}$ which is twice-continuously differentiable in $S$ satisfying the respective assumptions of the theorems and $\bar{x}_{p}$ is replaced with a non-degenerate local minimizer of $f$ (in fact, for Theorem \ref{theorem:convergence_general} an isolated local minimizer suffices).}
\begin{theorem}\label{theorem:rate_of_convergence}
Set $2\leq p<\infty$. Let $\bar{x}_{p}$ be the $L^{p}$ center of mass of $\{x_{i}\}_{i=1}^{N}\subset B(o,\rho)\subset M$, where $\rho\leq r_{\textrm{cx}}$. Suppose that $S$ is a strongly convex neighborhood around $\bar{x}_{p}$ in which $f_{p}$ is twice-continuously differentiable, and let $h_{S}$ and $H_{S}$, respectively, denote a lower and upper bound on the eigenvalues of the Hessian of $f_{p}$ in $S$. Furthermore, assume that $S$ is small enough such that one can choose $h_{S}>0$. In Algorithm \ref{table:T1} choose a constant step-size $t_{k}=t\in (0,\frac{2}{H_{S}})$. If after a finite number of iterations $k'$ each iterate continuously stays in $S$, then for $k\geq k'$ we have
\begin{equation}\label{eq:rete_of_conv}
d(x^{k},\bar{x}_{p})\leq K q^{\frac{k-k'}{2}}.
\end{equation}
In the above $K$ and $q$ are defined as
\begin{equation}\label{eq:c_q}
K=\bigg(\frac{2(f_{p}(x^{k'})-f_{p}(\bar{x}_{p}))}{h_{S}}\bigg)^{\frac{1}{2}} ~\textrm{and} ~q=1-\alpha(1-\frac{\alpha}{2})\frac{h_{S}}{H_{S}}(1+\frac{h_{S}}{H_{S}}),
\end{equation}
where $\alpha=tH_{S}$. In particular, $0\leq q<1$ and $x^{k}\rightarrow \bar{x}_{p}$ as $k\rightarrow \infty$.
\end{theorem}
\begin{proof}
Let $\gamma(t)=\exp_{x}(t\exp_{x}^{-1}\bar{x}_{p})$ be the minimal geodesic from $x\in S$ to $\bar{x}_{p}$. Note that $\gamma(t)\in S$ for $t\in [0,1]$ due to strong convexity of $S$. After writing the second order Taylor's series of $t\mapsto f_{p}(\gamma(t))$ around $t=0$ and using the bounds on the Hessian of $f_{p}$ one gets
\begin{equation}\label{eq:e1}
 -d(x,\bar{x}_{p})\|\nabla f_{p}(x)\|+\frac{h_{S}}{2}d^{2}(x,\bar{x}_{p}) \leq f_{p}(\bar{x}_{p})-f_{p}(x)\leq d(x,\bar{x}_{p})\|\nabla f_{p}(x)\|+\frac{H_{S}}{2}d^{2}(x,\bar{x}_{p}).
\end{equation}
Similarly by expansion of $t\mapsto f_{p}(\gamma(t))$ around $t=1$ one gets
\begin{equation}\label{eq:e2}
\frac{h_{S}}{2}d^{2}(x,\bar{x}_{p})\leq f_{p}(x)-f_{p}(\bar{x}_{p})\leq \frac{H_{S}}{2}d^{2}(x,\bar{x}_{p})
\end{equation}
Also notice that by the first order Taylor series expansion of $t\mapsto \nabla f_{p}(\gamma(t))$ around $t=1$ we have
\begin{equation}\label{eq:e3}
h_{S}d(x,\bar{x}_{p})\leq \|\nabla f_{p}(x)\|\leq H_{S} d(x,\bar{x}_{p})
\end{equation}
From the left inequality in (\ref{eq:e1}) and the right inequality in (\ref{eq:e2}) we have
\begin{equation}\label{eq:e4}
f_{p}(x)-f_{p}(\bar{x}_{p})\leq d(x,\bar{x}_{p})\|\nabla f_{p}(x)\|-\frac{h_{S}}{H_{S}}(f_{p}(x)-f_{p}(\bar{x}_{p})).
\end{equation}
Combining this with the left inequality in (\ref{eq:e3}) results in
\begin{equation}
h_{S}(1+\frac{h_{S}}{H_{S}}) \big(f_{p}(x)-f_{p}(\bar{x}_{p})\big)\leq \|\nabla f_{p}(x)\|^{2}.
\end{equation}
Now assume $k\geq k'$ so $x^{k},x^{k+1}\in S$. Then subtracting $f_{p}(\bar{x}_{p})$ from both sides of (\ref{eq:hessian_reduction}) and using (\ref{eq:e4}) both at $x=x^{k}$ yield
\begin{equation}
f_{p}(x^{k+1})-f_{p}(\bar{x}_{p})\leq q \big(f_{p}(x^{k})-f_{p}(\bar{x}_{p})\big).
\end{equation}
Therefore, we have
\begin{equation}
f_{p}(x^{k})-f_{p}(\bar{x}_{p})\leq \big(f_{p}(x^{k'})-f_{p}(\bar{x}_{p})\big) q^{k-k'},
\end{equation}
for $k\geq k'$. Now combining this with the left inequality in (\ref{eq:e2}) yields (\ref{eq:rete_of_conv}).
\end{proof}
This theorem predicts a \textit{lower bound} on the speed of convergence (i.e., the actual convergence is not worst than what the theorem predicts). The accuracy of this prediction, in part, depends on the accuracy of our estimates of the lower and upper bounds on the eigenvalues of the Hessian. Observe that when we are only given $H_{S}$, $\alpha=1$ (i.e., $t_{k}=\frac{1}{H_{S}}$) gives the smallest a-priori $q$. We call $t_{k}=\frac{1}{H_{S}}$ the \textit{best a-priori} step-size given $H_{S}$.

\begin{remark}[Asymptotic $q$]\label{rem:gradient_convexity}
Notice that a strongly convex $S$ which works in Theorem \ref{theorem:rate_of_convergence} might not work in Theorem \ref{theorem:convergence_general} (since the assumption $h_{S}>0$ might not hold true) and a smaller $S$ might be needed for this theorem. Nevertheless, if we start with an $S$ (and a corresponding $H_{S}$) for which Theorem \ref{theorem:convergence_general} holds, then there exists a smaller strongly convex $S'(\subset S)$ for which $h_{S'}>0$ and Theorem \ref{theorem:rate_of_convergence} holds. $H_{S}$ is still an upper bound on the eigenvalues of the Hessian of $f_{p}$ in $S'$ and in lack of any knowledge about $S'$ still $t=\frac{1}{H_{S}}$ is the best a-priori step-size. The actual asymptotic speed of convergence is determined by $q$ in a very small neighborhood $S'$ of $\bar{x}_{p}$. In fact, in the limit $h_{S'}$ and $H_{S'}$ converge, respectively, to $\lambda_{\textrm{min}}'$ and $\lambda_{\textrm{max}}'$ the smallest and largest eigenvalues of the Hessian of $f_{p}$ at $\bar{x}_{p}$. Therefore, for any step-size $t\in (0,\frac{2}{\lambda_{max}'})$, we define the associated \textit{asymptotic} $q$, denoted by $q'$, where in (\ref{eq:c_q}), $h_{S}$, $H_{S}$, and $\alpha$ are replaced, respectively, by $\lambda_{\textrm{min}}'$, $\lambda_{\textrm{max}}'$, and $\alpha'=t\lambda_{\textrm{max}}'$. Notice that if $t=\frac{1}{H_{S}}$, then $\alpha'<1$ and the smaller the $H_{S}$ the smaller the $q'$ will be. In general, ``smaller $H_{S}$'' means that either we make $S$ smaller or we choose a more accurate upper bound on the eigenvalues of the Hessian of $f_{p}$ in $S$.
\end{remark}
\subsection{A Conjecture: The Best Convergence Condition}\label{sec:Conjecture}
As mentioned before, reduction of the cost at each iteration is not enough to guarantee the convergence of Algorithm \ref{table:T1} to the global center of mass. Nevertheless, we conjecture that if the constant step-size is chosen not too large and the initial condition is not far from $\bar{x}_{p}$ (as specified next), then the cost $f_{p}$ can be reduced at each iteration, the iterates stay close to $\bar{x}_{p}$ and converge to it .
\begin{conjecture}\label{conj:optimal}
Set $p=2$ and let $\bar{x}_{2}$ be the $L^{2}$ center of mass of $\{x_{i}\}_{i=1}^{N}\subset B(o,\rho)\subset M$ where $\rho\leq r_{\textup{\textrm{cx}}}$. Let $H_{B(o,\rho)}={\textup{\textrm{c}}_{\delta}(2\rho)}$, where $\textup{\textrm{c}}_{\kappa}$ is defined in (\ref{eq:c_kappa}).
In Algorithm \ref{table:T1}, assume $x^{0}\in B(o,\rho)$ and choose a constant step-size $t_{k}=t$, for some 
$t\in (0, \frac{1}{H_{B(o,\rho)}}]$. Then we have the following: Each iterate continuously stays in $B(o,\rho)$ (and hence the algorithm will be well-defined for every $k\geq 0$), $f_{2}(x^{k+1})\leq f_{2}(x^{k})$ ($k\geq 0$) with equality only if $x^{k}=\bar{x}_{2}$,  and $x^{k}\rightarrow \bar{x}_{2}$ as $k\rightarrow \infty$. More generally, for $2\leq p<\infty$ the same results hold if $t\in (0, \frac{1}{H_{B(o,\rho),p}}]$, where $H_{B(o,\rho),p}=(2\rho)^{p-2}\max\{p-1,\textup{\textrm{c}}_{\delta}(2\rho)\}$.
\end{conjecture}
Now we explain the sense in which this conjecture is the best result one can hope for. We narrow down our desired class of convergence conditions to a class which gives conditions that are uniform in the data sets and in the initial condition. More specifically, we consider the following general and natural class of conditions:

\begin{quote}
Convergence Condition Class (C): Consider Algorithm \ref{table:T1} and fix $2\leq p<\infty$, and let $\delta$ and $\Delta$, respectively, be a lower and upper bound on sectional curvatures of $M$. Specify the largest $\bar{\rho}$ where $0<\bar{\rho}\leq r_{\textrm{cx}}$ such that for every $\rho\leq\bar{\rho}$ there are
\begin{enumerate}
\item  a number $\rho'_{\delta,\Delta,\rho,p}$ ($\rho\leq \rho'_{\delta,\Delta,\rho,p}\leq r_{\textrm{cx}})$ depending only on $\delta$, $\Delta$, $\rho$, and $p$; and
\item another number $t_{\delta,\Delta,\rho,\rho',p}$ depending only on  $\delta$, $\Delta$, $\rho$, $p$, and $\rho'_{\delta,\Delta,\rho,p}$,
\end{enumerate}
for which the following holds: for every ball $B(o,\rho)\subset M$, for every set of data points in $B(o,\rho)$, for every set of weights in (\ref{eq:fp}), for every initial condition in $B(o,\rho)$, and for every constant step-size $t_{k}=t\in (0,t_{\delta,\Delta,\rho,\rho',p}]$, each iterate of Algorithm \ref{table:T1} continuously stays in $B(o,\rho'_{\delta,\Delta,\rho,p})$ and $x^{k}\rightarrow \bar{x}_{p}$.
\end{quote}

First, notice that with $\bar{\rho}>r_\textrm{cx}$ there is no hope to have convergence to the global center (in this class), since the global center might not in $B(o,\bar{\rho})$ and in general $\nabla f_{p}$ might have more than one zero in $B(o,\bar{\rho})$. Next, observe that Conjecture \ref{conj:optimal} belongs to this class of conditions and it claims that $\bar{\rho}=r_{\textrm{cx}}$ is achievable; therefore, in this sense the conjecture claims the best possible condition in this class. In particular, this means that the conjecture gives the best possible spread of data points and the largest region of convergence, i.e., it allows both $\{x_{i}\}_{i=1}^{N}\subset B(o,r_{\textrm{cx}})$ and $x^{0}\in B(o,r_{\textrm{cx}})$.

Now let us see in what sense the step-size interval in Conjecture \ref{conj:optimal} is optimal. One can verify that $H_{B(o,\rho),p}$ in Conjecture \ref{conj:optimal} is the smallest \textit{uniform} upper bound on the eigenvalues of the Hessian of $f_{p}$ in $B(o,\rho)$. Here, by a uniform bound we mean a bound which is independent of the data points, the weights, and $o$. Based on Remark \ref{rem:gradient_convexity}, if we know that each iterate continuously stays in $B(o,\rho'_{\delta,\Delta,\rho,p})$ and further if we only know $H_{B(o,\rho'),p}$, then from Theorem \ref{theorem:rate_of_convergence} we see that $t_{k}=\frac{1}{H_{B(o,\rho'),p}}$ is the best uniform a-priori step-size (in the sense that it yields the smallest uniform a-priori $q$). Next, notice that, with this $t_{k}$, the smaller the $\rho'$, the smaller the $H_{B(o,\rho'),p}$ and hence the smaller the $q'$ (the asymptotic $q$) will be, see Remark \ref{rem:gradient_convexity}. Consequently, $t_{k}=\frac{1}{H_{B(o,\rho'),p}}$ at $\rho'=\rho$ gives the smallest \emph{asymptotic} $q$ among all  \emph{best uniform a-priori} step-sizes $t_{k}=\frac{1}{H_{B(o,\rho'),p}}$, where $\rho\leq \rho'\leq r_{\textrm{cx}}$. Conjecture \ref{conj:optimal} claims that, indeed, $\rho'_{\delta,\Delta,\rho,p}$ can be as small as $\rho$ (independent of $\delta,\Delta, p$). Therefore, in summary, among all conditions in class (C), the sub-class which prescribes $t_{\delta,\Delta,\rho,\rho',p}=\frac{1}{H_{B(o,\rho'_{\delta,\Delta,\rho,p}),p}}$ allows to have the smallest uniform a-priori $q$ (by choosing $t_{k}=\frac{1}{H_{B(o,\rho'_{\delta,\Delta,\rho,p}),p}}$), and in this sub-class, Conjecture \ref{conj:optimal} gives the largest $t_{\delta,\Delta,\rho,\rho',p}$, therefore it allows for the largest step-size and hence the \textit{smallest uniform asymptotic} $q$ in this sub-class. We stress that this sense of optimality of the step-size interval should not be construed as giving the best speed of convergence for any actual data configuration; rather it gives the best uniform lower bound on the speed of convergence in Theorem \ref{theorem:rate_of_convergence}.  This means that for all data configurations, weights, and initial conditions in $B(o,\rho)$ the actual speed of convergence will not be worse than the one predicted by Theorem \ref{theorem:rate_of_convergence} where the associated $q$ in (\ref{eq:c_q}) is the best uniform a-priori $q$. Quite similarly, one could argue that the step-size interval in Conjecture \ref{conj:optimal} is optimal in the sense that allows for the most reduction per iteration in the upper bound given in (\ref{eq:hessian_reduction}) of Proposition \ref{prop:hessian_stepsize}.

The proof of the conjecture in the case of manifolds with zero curvature is quite easy and straightforward. However, the general case seems to be difficult. The main difficulty in proving Conjecture \ref{conj:optimal} is in proving that the iterates continuously stay in $B(o,\rho)$. Nevertheless, in Theorem \ref{theorem:sphere_L2} we prove the conjecture for manifolds of constant nonnegative curvature.  As this proof suggests, we believe that the difficulty in proving this conjecture has more to do with geometry (than optimization) and the need of good estimates (which currently seem not to exist) about the behavior of the exponential map in a manifold. Our proof of Theorem \ref{theorem:sphere_L2} certainly constitutes some strong evidence that the conjecture also is true for  manifolds of nonnegative curvature. For manifolds with negative curvature we have also some evidence that the conjecture is true. For example, the conjecture  is trivially true if all the data points are concentrated at a single point in $B(o,\rho)$, and by continuity, it is also true if all the data points are concentrated enough around a single point in $B(o,\rho)$. Weaker convergence results can be established with some efforts. For example, in Section \ref{sec:sub_optimal}, we derive weaker convergence results in Theorems \ref{theorem:compromised_L2} and \ref{theorem:compromise_step_size_L2}. As a comparison, Theorem \ref{theorem:compromised_L2} gives smaller allowable spread and smaller region of convergence than Conjecture \ref{conj:optimal}. Theorem \ref{theorem:compromise_step_size_L2}, on the other hand, gives allowable spread and region of convergence which could be very close to $B(o,r_{\textrm{cx}})$, but the allowable step-size is restricted significantly in this theorem.

\begin{remark}\label{rem:step_size}
It is interesting to note that except for the $L^{2}$ center of mass in a manifold of nonnegative curvature, in all other cases the best step-size depends on $p$ and the radius $\rho$, where the latter in a typical scenario is not often known a-priori or at least its estimation requires further processing. Moreover, for manifolds of negative curvature a lower bound on the curvature is also needed to determine the best step-size. Fortunately, the class of manifolds with nonnegative curvature already covers a big bulk of applications, and since these manifolds are usually compact an a-priori upper bound on $\rho$ (e.g., the diameter of the manifold or $r_{\textrm{cx}}$) can be used to determine a step-size (obviously, not the best one) for values of $p$ other than $2$.
\end{remark}
\begin{remark}[Related to Remark \ref{rem:gradient_convexity}]\label{rem:convergence_behavior}
It is useful to put this conjecture in some context, especially in view of Theorems \ref{theorem:convergence_general} and \ref{theorem:rate_of_convergence} and Remark \ref{rem:gradient_convexity}. It should be clear from our discussions in Subsection \ref{sec:Lp_center} and Remark \ref{rem:gradient_convexity}, that when $\Delta>0$, the conjecture claims convergence for initial conditions in regions in which the Hessian of $f_{p}$ is not necessarily positive-definite. Therefore, what really could help us in proving this result is Theorem \ref{theorem:convergence_general} and not Theorem \ref{theorem:rate_of_convergence}. Although, already assured of convergence, we can use Theorem \ref{theorem:rate_of_convergence} to give us an asymptotic behavior of the algorithm. All our proved convergence theorems (which are Theorems \ref{theorem:sphere_L2}, \ref{theorem:compromised_L2}, and \ref{theorem:compromise_step_size_L2}) are proved using Theorem \ref{theorem:convergence_general}. In Theorems \ref{theorem:sphere_L2} and \ref{theorem:compromised_L2}  both the initial conditions and the trajectories of the algorithm can lie in regions in which only this theorem applies. The case of Theorem \ref{theorem:compromised_L2} is rather interesting. Under the conditions of Theorem \ref{theorem:compromised_L2} the initial condition must lie in a region in which $f_{p}$ happen to be strictly convex (since $\frac{1}{3}r_{\textrm{cx}}< \frac{\pi}{4\sqrt{\Delta}}$, see Subsection \ref{sec:Lp_center}), however, the trajectory of the algorithm can visit a region in which the Hessian of $f_{p}$ is not positive-definite.
\end{remark}

\begin{remark}\label{rem:cont_stay}
From (\ref{eq:grad_Lp}) we have $\|\nabla f_{p}(x)\|<(2\rho)^{p-1}$ for $x\in B(o,\rho)$, and we can easily verify that $H_{B(o,\rho),p}\geq (2\rho)^{p-2}$. Consequently, for each $x^{k}$ we have $t_{k}\|\nabla f_{p}(x^{k})\|<2\rho<\textrm{inj}M$. Therefore, the conditions of Proposition \ref{prop:cont_stay} are satisfied, and we could drop the adverb \textit{continuously} in the statement of the conjecture. It is interesting to note that, this would have not been the case if we had the condition $t\in (0,  \frac{2}{H_{B(o,\rho),p}})$ (which is enough for reduction at each iteration but not enough for continuous stay in $B(o,\rho)$). As another example, in Theorem \ref{theorem:compromised_L2}, since we allow $t\in (0,\frac{2}{H_{B(o,\rho),p}})$ we do not have the equivalence between ``staying'' and ``continuously staying.''
\end{remark}

\begin{remark}[Gradient vs. Newton's]\label{rem:Grad_vs_Newton}
According to this conjecture the domain of convergence of gradient descent for finding the center of mass is relatively large, in the sense that the algorithm can start from any point in the same ball that contains the data points, and moreover, the radius of the ball is not restricted more than what the existence and uniqueness Theorem \ref{theorem:exist_unique} requires. In comparison, if we consider implementing Newton's method for finding the center of mass (as for example in \cite{Groisser}), we do not expect to have such a large domain of convergence. We know that, in general, the gradient descent method has a larger convergence domain compared with Newton's method, but still one might have slight hope that one could prove such a large domain of convergence for Newton's method applied to our problem. However, as mentioned in Subsection \ref{sec:Lp_center}, if $\Delta>0$ and $\frac{\pi}{4\sqrt{\Delta}}<\rho\leq r_{\textrm{cx}}$, then the Hessian of $f_{p}$ might be in-definite at some points in $B(o,\rho)$. This fact, ruins our hope for having Newton's method with a domain of convergence as large as the gradient method, since Newton's method involves inverting (in an appropriate sense) the Hessian operator (see \cite{Groisser}, \cite{AbsMahSep2008}, or \cite{Udriste_Book} for more details). In fact, this suggests that, in a manifold of positive curvature, even implementing Newton's method for data points which are spread in a ball of radius larger than $\frac{\pi}{4\sqrt{\Delta}}$ might be impossible or very difficult. In fact, results in \cite{Groisser} (in which the estimated domain of convergence for Newton's method lie in balls of radius smaller than $\frac{\pi}{4\sqrt{\Delta}}$) also corroborate this observation.
\end{remark}

\subsection{Prior work}\label{sec:prior}
There are not many \textit{accurate} and \textit{correctly proven}  \footnote{A mistake made by some authors (see e.g., \cite{Manton_globally} and \cite{Fletcher_median}) in proving such results has been to wrongly assume that $f_{p}:M\rightarrow \mathbb{R}$ is \textit{globally} convex (or strictly convex) if the data points are in a small enough ball, which in general is not true, e.g., if  $M$ is compact (see Theorem \ref{theorem:Yau}).}results available about the convergence of gradient descent (and more specifically constant-step-size) for locating the Riemannian center of mass. For constant step-size algorithms, the most accurate and useful results are due to Le \cite{Le1} and Groisser \cite{Groisser,Groisser2} for the $L^{2}$ mean. Both authors show that if the data points are concentrated enough then the map $x\mapsto \exp_{x}(-\nabla f_{2}(x))$, when  \emph{restricted} to a small neighborhood of $\bar{x}_{2}$, is a contraction mapping (with fixed point $\bar{x}_{2}$); and using this they prove a convergence result. Groisser's results are more general and allow to analyze both the Newton's method and the gradient descent method, while Le's result gives a slightly better bound on the allowable spread of the data points. Le shows that when $M$ is a locally symmetric  manifold of nonnegative sectional curvature and $\rho\leq \frac{3}{10}r_{\textrm{cx}}$, then with step-size $t_{k}=1$ and starting from the \textit{center} of the ball $B(o,\rho)$, Algorithm \ref{table:T1} locates the global Riemannian mean (see Corollary 2 in \cite{Le1}). In fact, Le shows that the map $x\mapsto \exp_{x}(-\nabla f_{2}(x))$ is a contraction mapping when restricted to $B(\bar{x}_{2},\rho)$. Since $\bar{x}_{2}$ is a fixed point of the map, \textit{starting} from $o$ the iterates will not leave $B(\bar{x}_{2},\rho)$ (but not $B(o,\rho)$, necessarily). Le's result leaves room for significant improvement in the allowable spread of data points compared to our Conjecture \ref{conj:optimal}. Notice that Le's result can be used to deduce convergence for an \textit{arbitrary} initial condition in $B(o,\rho)$ assuming a $\rho$ half as before, that is $\rho\leq \frac{3}{20}r_{\textrm{cx}}$. This is obviously a more practical scenario. Our Theorem \ref{theorem:compromised_L2} (which needs only $\rho\leq \frac{1}{3}r_{\textrm{cx}}$ and does not require local symmetry or nonnegative curvature) is a considerable improvement over Le's result. Still our Theorem \ref{theorem:sphere_L2}, which is the best one can hope for in the case of manifolds of constant nonnegative curvature, is an even further improvement over Le's result (when applied to these manifolds).

In \cite{Le_barycenters_2004} a convergence result is given for a (hard-to-implement) gradient method which varies the step-size in order to confine the iterates to a small ball. A result in \cite{PhDKrakowski} is somewhat similar in nature to our Theorem \ref{theorem:compromised_L2}, yet it does not yield an explicit convergence condition. Local convergence of Algorithm \ref{table:T1} with $t_{k}=1$ on $\mathbb{S}^{n}$ under the generic condition of ``$x^{0}$ being close enough to the center'' is argued in \cite{Buss}; however, such a condition is of little practical use. A few authors have also studied other related problems and methods e.g.,  stochastic gradient methods \cite{Arnaudon1}, projected gradient methods \cite{Karkowki_Manton}, Newton's method \cite{Buss} and \cite{Groisser}, and variable step-size gradient algorithm for the $L^{1}$ mean or median \cite{Yang1}.

\section{Convergence on Manifolds of Constant Nonnegative Curvature (An Optimal Result)}\label{sec:optmial}
In this section, we prove Theorem \ref{theorem:sphere_L2} which is essentially Conjecture \ref{conj:optimal} for a manifold of constant nonnegative curvature.\footnote{To be accurate, Theorem \ref{theorem:sphere_L2} is little bit more than Conjecture \ref{conj:optimal} as it also states a result about the iterates getting trapped in the convex hull of the data points. } 
\subsection{A useful triangle secant comparison result}\label{sec:comparison}
Here, we prove the comparison result used to prove Theorem \ref{theorem:sphere_L2} (see also Figure \ref{fig:secant} and Theorem \ref{theorem:convex_combination_const_pos}). Referring to Figure \ref{fig:secant}, the theorem is about comparing the lengths of corresponding geodesic secants $xm$ and $\tilde{x}\tilde{m}$ of two (geodesic) triangles $\triangle y_{1}xy_{2}$ and $\triangle \tilde{y}_{1}\tilde{x}\tilde{y}_{2}$ which are in $M=\mathbb{S}^{2}$ and $\mathbb{R}^{2}$, respectively, and in which the angles at $x$ and $\tilde{x}$ and their corresponding sides are equal.  The bi-sectors of angles $\angle y_{1}xy_{2}$ and $\angle \tilde{y}_{1}\tilde{x}\tilde{y}_{2}$ are examples of such secants and the theorem implies that the bi-sector of $\angle y_{1}xy_{2}$ is larger than that of $\angle \tilde{y}_{1}\tilde{x}\tilde{y}_{2}$. In general, it should be obvious that such a comparison is meaningful only if $M$ is a manifold of constant (sectional) curvature or if $M$ is two-dimensional. In the case of a constant curvature manifold $M$ the enabling property is the well-known \textit{axiom of plane}: Let $x\in M$ and assume that $W\subset T_{x}M$ is a $k$-dimensional subspace of $T_{x}M$, then the set $\exp_{x}(W\cap B(0_{x},\rho))$ is a totally geodesic submanifold of $W$. Here $B(0_{x},\rho)$ is the open ball of radius $\rho$ around the origin of $T_{x}M$ and $0<\rho<\textrm{inj}M$ (see e.g., \cite[p. 136]{Sakai}).

It seems that this kind of comparison cannot be deduced, at least immediately, from standard Toponogov's comparison theorems (see e.g., \cite{Chavel}, \cite{Sakai}, \cite{Cheeger_Comparison}, and \cite{Karcher2}). Although we are almost certain that this result has been known before, we were not able to find either a proof for it or even its statement. Our proof here is a direct one and is not in the spirit of standard comparison theorems based on comparison of solutions of two differential equations.

\begin{figure}[h]
  \begin{center}
  \scalebox{.75}{\input{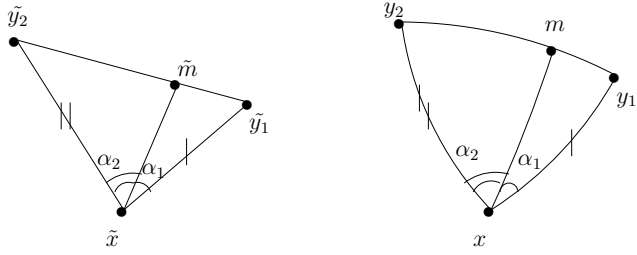}}\\
  \caption{$\triangle y_{1}xy_{2}$ is a triangle in a manifold of constant positive curvature and $\triangle\tilde{y}_{1}\tilde{x}\tilde{y}_{2}$ is the corresponding triangle in Euclidean space. Corresponding equal angles and sides are marked. According to Theorem \ref{theorem:comparison_sphere}, the geodesic secant $xm$ is longer than the secant $\tilde{x}\tilde{m}$.
  }\label{fig:secant}
\end{center}
\end{figure}
\begin{theorem}\label{theorem:comparison_sphere}
Consider a triangle with vertices $x,y_{1},y_{2}\in M=\mathbb{S}^{2}$ (or $\mathbb{S}^{n}$) and with minimal geodesic sides $xy_{1}$, $xy_{2}$, $y_{1}y_{2}$ of lengths $b$, $c$, and a respectively, where $a+b+c<2\pi$. Assume that the internal angle $\angle y_{1}xy_{2}$ is equal to $\alpha$ ($0<\alpha\leq\pi$). Consider another triangle in $\mathbb{R}^{2}$ (or $\mathbb{R}^{n}$) with vertices $\tilde{x},\tilde{y}_{1},\tilde{y}_{2}$ and corresponding side lengths $\tilde{a}$ and $\tilde{b}$ such that $\tilde{b}=b$, $\tilde{c}=c$ and $\angle\tilde{y_{1}}\tilde{x}\tilde{y}_{2}=\alpha$. Consider a geodesic secant $\gamma$ in triangle $y_{1}xy_{2}$ passing through $x$ making angles $\alpha_{1}$ and $\alpha_{2}$ ($\alpha_{1}+\alpha_2=\alpha$) with minimal geodesic sides $xy_{1}$ and $xy_{2}$, respectively. Denote by $m$ the point where the geodesic secant meets the minimal geodesic side $y_{1}y_{2}$ for the first time.\footnote{Both $\alpha =\alpha_{1}+\alpha_{2}$ and the secant meeting the minimal geodesic side $y_{1}y_{2}$ follow from the axiom of plane.}Similarly, let a secant line of triangle $\tilde{y}_{1}\tilde{x}\tilde{y}_{2}$ passing through $\tilde{x}$ make angles $\alpha_{1}$ and $\alpha_{2}$ with sides $\tilde{x}\tilde{y}_{1}$ and $\tilde{x}\tilde{y}_{2}$, respectively, and denote by $\tilde{m}$ the point where this secant line meets the side $\tilde{y}_{1}\tilde{y}_{2}$. Then we have that the length of the secant segment $xm$ is larger than or equal to the length of the secant segment $\tilde{x}\tilde{m}$ with equality if and only if $\alpha_1=0$, $\alpha_2=0$, $b=0$, $c=0$, or $\alpha =\pi$. More generally, if $M$ is a manifold of constant curvature $\Delta\geq 0$, the same result holds if $x,y_{1},y_{2}$ belong to a ball of radius not larger than $r_{\textrm{cx}}$. 
\end{theorem}

\begin{proof}
Denote the length of the geodesic secant segment $xm$ by $z(b,c;\alpha_{1},\alpha_{2})$. Using spherical trigonometric identities (e.g., \cite[p. 53]{treatise_spherical}) one can show that (see also \cite[p. 55]{treatise_spherical})
\begin{equation}\label{eq:cot_z_sphere}
\cot{z(b,c;\alpha_1,\alpha_2)}=\frac{\cot{b}\sin{\alpha_{2}}+\cot{c}\sin{\alpha_{1}}}{\sin(\alpha_{1}+\alpha_{2})}.
\end{equation}
Similarly, denote the length of the secant segment $\tilde{x}\tilde{m}$ by $\tilde{z}(b,c,\alpha_{1},\alpha_{2})$. It is easy to see that
\begin{equation}\label{eq:cot_z_plane}
\tilde{z}(b,c;\alpha_{1},\alpha_{2})=\frac{bc~\sin(\alpha_{1}+\alpha_{2})}{b\sin{\alpha_{1}}+c\sin{\alpha_{2}}},
\end{equation}
where in both relations $\alpha_{1}+\alpha_{2}=\alpha$. The claim is obvious for $\alpha_{1}=0$, $\alpha_{2}=0$, $b=0$, or $c=0$. Also from $b+c+a<2\pi$ one can show that $\alpha=\pi$ only if $a=b+c$ in which case again we have equality. Note that $z$ and $\tilde{z}$ are both smaller than $\pi$ and therefore to show $z(b,c;\alpha_{1},\alpha_{2})>\tilde{z}(b,c;\alpha_{1},\alpha_{2})$ we could show $\cot{z}(b,c;\alpha_{1},\alpha_{2})<\cot{\tilde{z}(b,c;\alpha_{1},\alpha_{2})}$ with $\alpha_{1}+\alpha_2=\alpha$ (see (\ref{eq:cot_z_sphere}) and (\ref{eq:cot_z_plane})). First, we note that $t\mapsto g(t)=\cot\frac{1}{t}$ is strictly concave in the interval $(\frac{1}{\pi},\infty)$. To see this notice that $g$ is smooth and its second derivative is
\begin{equation}
g''(t)=-\frac{2}{t^{3}}(1+\cot^{2}\frac{1}{t})(1-\frac{\cot{\frac{1}{t}}}{t}),
\end{equation}
which is strictly negative. Note that in the above, the last term is positive in $(\frac{1}{\pi},\infty)$ because $t\cot{t}<1$ in $(0,\pi)$. Now set $t_{1}=\frac{1}{b}$, $t_{2}=\frac{1}{c}$, $\lambda_{1}=\frac{\sin(\alpha-\alpha_{1})}{\sin(\alpha-\alpha_{1})+\sin{\alpha_{1}}}$ and $\lambda_{2}=\frac{\sin\alpha_{1}}{\sin(\alpha-\alpha_{1})+\sin{\alpha_{1}}}=1-\lambda_{1}$. Notice that $t_{1},t_{2},\lambda_{1}t_{1}+\lambda_{2}t_{2}>\frac{1}{\pi}$ (because $0\leq b,c<\pi$). From $\lambda_{1}g(t_{1})+\lambda_{2}g(t_{2})\leq g(\lambda_{1}t_1+\lambda_{2}t_2)$ in which equality is only if $t_{1}=t_{2}$ (or equivalently if $b=c$) we have
\begin{equation}\label{eq:cot_big}
\frac{\sin(\alpha-\alpha_{1})\cot{b}+\sin\alpha_{1}\cot{c}}{\sin(\alpha-\alpha_{1})+\sin{\alpha_{1}}}\leq \cot\bigg(\frac{bc~(\sin(\alpha-\alpha_{1})+\sin{\alpha_{1}})}{b\sin{\alpha_{1}}+c\sin(\alpha-\alpha_{1})}\bigg).
\end{equation}
It is easy to verify that $\cot{t}\leq s \cot{(ts)}$ for $0< t<\pi$ and $0< s\leq 1$ with equality only if $s=1$. Without loss of generality we can assume $0\leq \alpha_{1}\leq\frac{\alpha}{2}$, and hence for $s= \frac{\sin{\alpha}}{\sin(\alpha-\alpha_{1})+\sin{\alpha_{1}}}$, we have $0< s\leq 1$ with $s=1$ only if $\alpha_{1}=0$. Consequently, by applying the mentioned fact to the right hand side of (\ref{eq:cot_big}) with the given $s$ we have (assuming $\alpha_{1}>0$)
\begin{equation}
\frac{\sin(\alpha-\alpha_{1})\cot{b}+\sin\alpha_{1}\cot{c}}{\sin(\alpha-\alpha_{1})+\sin{\alpha_{1}}}<
\cot\big(\frac{bc~\sin{\alpha}}{b\sin{\alpha_{1}}+c\sin(\alpha-\alpha_{1})}\big)~\frac{\sin{\alpha}}{\sin(\alpha-\alpha_{1})+\sin{\alpha_{1}}},
\end{equation}
which means that $z(b,c,\alpha_{1},\alpha-\alpha_{1})>\tilde{z}(b,c,\alpha_{1},\alpha-\alpha_{1})$.
\end{proof}

\subsection{Riemannian convex combinations}\label{sec:convex_comb}
Intuitively, one would like to think of $\exp_{x}(\sum_{i=1}^{N}w_{i}\exp_{x}^{-1}x_{i})$ as a Riemannian convex combination with similar properties as the Euclidean convex combination. We call this a convex combination of $\{x_{i}\}_{i=1}^{N}\subset M$ with respect to $x\in M$ and with weights $(w_{1},\ldots,w_{N})$. \footnote{Notice that if $M=\mathbb{R}^{n}$, $\exp_{x}(\sum_{i=1}^{N}w_{i}\exp_{x}^{-1}x_{i})$ translates to $x+\sum_{i=1}^{N}w_{i}(x_{i}-x)$, which is independent of $x$. In a nonlinear space the convex combination does not enjoy this base-point independence and that is why we have been explicit in calling $\exp_{x}(\sum_{i=1}^{N}w_{i}\exp_{x}^{-1}x_{i})$ a convex combination with respect to $x$. Moreover, to have ``nice'' properties, $x$ cannot be arbitrary and must belong the convex hull of $\{x_{i}\}_{i=1}^{N}$, as explained in Remark \ref{rem:conx_comb_wrt}.}  For the case of a manifold of constant nonnegative curvature we show that this is a valid definition. The following general proposition is very useful in making sense of Riemannian convex combination as well as proving Theorem \ref{theorem:sphere_L2}.
\begin{proposition}\label{prop:convex_combination}
Let $S\subset M$ be a strongly convex set containing $\{x_{i}\}_{i=1}^{N}$ and let $x\in S$. Assume that for arbitrary weights $0\leq w_{1}, w_{2}\leq 1$ (with $w_{1}+w_{2}=1$) and for any $y_{1},y_{2}\in S$, $\exp_{x}\left(t( w_{1}\exp_{x}^{-1}y_{1}+w_{2}\exp_{x}^{-1}y_{2})\right)\in S$ for $t\in[0,1]$. Then for every set of points $\{x_{i}\}_{i=1}^{N}\subset S$ and corresponding weights $0\leq w_{i}\leq 1$ (with $\sum_{i}w_{i}=1$), $\exp_{x}(t\sum_{i=1}^{N}w_{i}\exp_{x}^{-1}x_{i})$ also belongs to $S$ for $t\in[0,1]$.
\end{proposition}
\begin{proof}
We prove the claim for $N=3$ and for larger $N$ it follows by induction. 
Let $y(t)=\exp_{x}(t\sum_{i=1}^{3}w_{i}\exp_{x}^{-1}x_{i})$. Note that we can write
\begin{equation}
y(t)=\exp_{x}\left(t\bigg(w_{1}\exp_{x}^{-1}x_{1}+(1-w_{1})\big(\frac{w_{2}}{w_{2}+w_{3}}\exp_{x}^{-1}x_{2}+
\frac{w_{3}}{w_{2}+w_{3}}\exp_{x}^{-1}x_{3}\big)\bigg)\right).
\end{equation}
Since by assumption $\tilde{x}_{2}=\exp_{x}(\frac{w_{2}}{w_{2}+w_{3}}\exp_{x}^{-1}x_{2}+
\frac{w_{3}}{w_{2}+w_{3}}\exp_{x}^{-1}x_{3})$ belongs to $S$, there exists a unique minimizing geodesic between $x$ and $\tilde{x}_{2}$. Therefore,
$\exp_{x}^{-1}\tilde{x}_{2}$ is well-defined and belongs to the injectivity domain of $\exp_{x}$ and we have $\exp_{x}^{-1}\tilde{x}_{2}=\frac{w_{2}}{w_{2}+w_{3}}\exp_{x}^{-1}x_{2}+
\frac{w_{3}}{w_{2}+w_{3}}\exp_{x}^{-1}x_{3}$. On the hand, by our assumption $\exp_{x}\big(t(w_{1}\exp_{x}^{-1}x_{1}+(1-w_{1})\exp_{x}^{-1}\tilde{x}_{2})\big)$ belongs to $S$, which means that $y(t)\in S$ for $t\in[0,1]$.
\end{proof}

An example of a set $S$ is the convex hull of $\{x_{i}\}_{i=1}^{N}$. Recall that the convex hull of $A\subset M$ (if it exists) is defined as the smallest strongly set containing $A$. If $A$ lies in a strongly convex set obviously its convex hull exists. It is known that the convex hull of a finite set of points in a constant curvature manifold is a closed set. Also in a manifold of constant curvature the $L^{p}$ center of mass ($1<p<\infty$) of $\{x_{i}\}_{i=1}^{N}$ with weights $w_{i}\geq 0$ belongs to the convex hull of $\{x_{i}\}_{i=1}^{N}$ and if $w_{i}>0$ for every $i$, it belongs to the interior of the hull \cite{BijanLp}.

The next theorem makes exact sense of an intuitive notion of Riemannian convex combination in a manifold of constant nonnegative curvature:
\begin{theorem}\label{theorem:convex_combination_const_pos}
Let $M$ be a Riemannian manifold of constant non-negative curvature and $S$ a strongly convex set containing $\{x_{i}\}_{i=1}^{N}$. Assume that $S$ lies in a ball of radius of at most $r_{\textup{\textrm{cx}}}$. For every $x\in S$ and arbitrary weights $w_{i}\geq 0$ (with $\sum_{i=1}^{N}w_{i}=1$), we have $\exp_{x}(t\sum_{i=1}^{N}w_{i}\exp_{x}^{-1}x_{i})\in S$ for $t\in [0,1]$. In particular, if $S$ is the convex hull of $\{x_{i}\}_{i=1}^{N}$, then the Riemannian convex combination $\exp_{x}(\sum_{i=1}^{N}w_{i}\exp_{x}^{-1}x_{i})$ belongs to $S$ for every $x\in S$.
\end{theorem}
\begin{proof}
By Proposition \ref{prop:convex_combination}, it suffices to show that for arbitrary $y_{1},y_{2}\in S$ and weights $(w_{1},w_{2})$ with $w_{1}+w_{2}=1$ we have $\tilde{\tilde{m}}(t)=\exp_{x}\big(t(w_{1}\exp_{x}^{-1}y_{1}+w_{2}\exp_{x}^{-1}y_{2})\big)\in S$ for $\in [0,1]$. This follows from comparison Theorem \ref{theorem:comparison_sphere}. To see this, first note that, in triangle $xy_{1}y_{2}$ in Figure \ref{fig:secant}, there is a $1-1$ correspondence between the weight pairs $(w_{1},w_{2})$ and the angle pairs $(\alpha_{1},\alpha_{2})$, where $\alpha_{1}+\alpha_{2}=\alpha$ (the axiom of plane is essential for this fact to hold). If $x,y_{1},y_{2}\in S$, then by strong convexity of $S$ we have $m\in S$. Since the distance between $x$ and $\tilde{\tilde{m}}(1)$ is nothing but the length of secant $\tilde{x}\tilde{m}$ in triangle $\tilde{x}\tilde{y}_{1}\tilde{y}_{2}$, it follows from Theorem \ref{theorem:comparison_sphere} that $\tilde{\tilde{m}}(t)$ must belong to $S$ for $t\in [0,1]$.
\end{proof}
\begin{remark}\label{rem:conx_comb_wrt}
There is a subtle difference between the convex combination in Euclidean space and in a manifold of constant nonnegative curvature. In Euclidean space $\exp_{x}(\sum_{i=1}^{N}w_{i}\exp_{x}^{-1}x_{i})$ belongs to the convex hull of $\{x_{i}\}_{i=1}^{N}$ for an arbitrary $x$ (i.e., inside or outside the convex hull). However, it follows from Theorem \ref{theorem:comparison_sphere} that, in the case of a manifold of constant nonnegative curvature, if $x$ is not in the convex hull of $\{x_{i}\}_{i=1}^{N}$, then $\exp_{x}(\sum_{i=1}^{N}w_{i}\exp_{x}^{-1}x_{i})$  also does not belong to the convex hull, necessarily.
\end{remark}
\begin{remark}\label{rem:conv_neg_curv}
One might wonder in what directions the above theorem can be extended.  One can show that in a manifold of constant \textit{negative} curvature the inequality in Theorem \ref{theorem:comparison_sphere} holds in the reverse direction, that is, the secant in the manifold is shorter than the corresponding secant in $\mathbb{R}^{n}$. This by itself implies that, in a manifold of constant negative curvature, $\exp_{x}(\sum_{i=1}^{N}w_{i}\exp_{x}^{-1}x_{i})$  \textit{does not} belong to the convex hull of $\{x_{i}\}_{i=1}^{N}$, necessarily; and one needs to scale down the tangent vector $\sum_{i=1}^{N}w_{i}\exp_{x}^{-1}x_{i}$ to ensure it belong to the convex hull. This scaling somehow should be related to the size of the convex hull or the minimal ball of $\{x_{i}\}_{i=1}^{N}$ (see Conjecture \ref{conj:optimal} and Remark \ref{rem:Groisser}). Recall that the \textit{minimal ball} of $\{x_{i}\}_{i=1}^{N}\subset M$  is a closed ball of minimum radius containing $\{x_{i}\}_{i=1}^{N}$ (see \cite{BijanLp} for more on the minimal ball). Furthermore, even for nonnegative variable curvature, $\exp_{x}(\sum_{i=1}^{N}w_{i}\exp_{x}^{-1}x_{i})$ belonging to the convex hull of $\{x_{i}\}_{i=1}^{N}$ seems implausible. However, in this case, we conjecture that $\exp_{x}(\sum_{i=1}^{N}w_{i}\exp_{x}^{-1}x_{i})$ belongs to the minimal ball of $\{x_{i}\}_{i=1}^{N}$. Also another curious question whose answer is most likely ``yes'' is: ``In a manifold of constant nonnegative curvature, can one generate the entire convex hull of $\{x_{i}\}_{i=1}^{N}$ by varying the weights in $\exp_{x}(\sum_{i=1}^{N}w_{i}\exp_{x}^{-1}x_{i})$, where $x$ belongs to the convex hull?''
\end{remark}

\subsection{Convergence result}\label{sec:converge_optimal}
 We are now ready to state and prove the main theorem of the section.
\begin{theorem}\label{theorem:sphere_L2}
Assume that $M$ is a complete Riemannian manifold with constant nonnegative sectional curvature $\Delta=\delta\geq 0$. Let $p=2$ and $\{x_{i}\}_{i=1}^{N}\subset B(o,\rho)$, where $\rho\leq r_{\textup{\textrm{cx}}}$ (see (\ref{eq:rcx})). In Algorithm \ref{table:T1}, choose an initial point $x^{0}\in B(o,\rho)$ and a constant step-size $t_{k}=t$, where $t\in(0,1]$. Then we have: The algorithm is well-defined for every $k\geq 0$, each iterate continuously stays in $B(o,\rho)$, $f_{2}(x^{k+1})\leq f_{2}(x^{k})$ with equality only if $x_{k}=\bar{x}_{2}$, and $x^{k}\rightarrow \bar{x}_{2}$ as $k\rightarrow \infty$. Moreover, if for some $k'\geq 0$, $x^{k'}$ belongs to the convex hull of $\{x_{i}\}_{i=1}^{N}$, then $x^{k}$ also belongs to the convex for $k\geq k'$. More generally, for $2\leq p<\infty$ the same results hold if we take $t\in (0,t_{\rho,p}]$ with $t_{\rho,p}=\frac{1}{H_{B(o,\rho),p}}$ where $H_{B(o,\rho),p}=(p-1)\big(2\rho\big)^{p-2}$.
\end{theorem}
\begin{proof}
The fact that each iterate continuously stays in $B(o,\rho)$ follows from Theorem \ref{theorem:convex_combination_const_pos}. The same argument shows that if $x^{k'}$ is in the convex hull of $\{x_{i}\}_{i=1}^{N}$, then $x^{k}$ also belongs to the hull for $k\geq k'$. By Proposition \ref{prop:hessian_stepsize}, step-size $t_{k}=t$ at each step results in strict reduction of $f_{2}$ unless at $\bar{x}_{2}$. Next, the iterates converging to $\bar{x}_{2}$ follows from Theorem \ref{theorem:convergence_general} by taking $S$ as $B(o,\rho)$ or the convex hull of $\{x_{i}\}_{i=1}^{N}$. For the general $p$, we notice that
$\frac{-1}{H_{B(o,\rho),p}}\nabla f_{p}(x)$ can be written as $\sum_{i=1}^{N}\tilde{w}_{i}\exp_{x}^{-1}x_{i}$, where
$\sum_{i=1}^{N}\tilde{w}_{i}\leq 1$ and $\tilde{w}_{i}\geq 0$. Therefore, again we can use Theorem \ref{theorem:convex_combination_const_pos}, and the rest of the claims follow similarly.
\end{proof}

\begin{remark}\label{remark:conjecture}
In a manifold of non-constant curvature, Theorem \ref{theorem:comparison_sphere} is not meaningful; therefore, a direct application of this theorem to prove Conjecture \ref{conj:optimal} for an arbitrary manifold of nonnegative curvature is not plausible. However, it is likely that comparison Theorem \ref{theorem:comparison_sphere} might be used to prove a useful comparison theorem for the case of manifold of variable nonnegative curvature, which in turn could allow us to apply Proposition \ref{prop:convex_combination}.
\end{remark}

\begin{remark}\label{rem:Groisser}
In \cite{Groisser}, Groisser introduced the notion of tethering: A map $\Psi: M\rightarrow M$ is called \emph{tethered} to $\{x_{i}\}_{i=1}^{N}$ if for every strongly convex regular geodesic ball $B$ containing $\{x_{i}\}_{i=1}^{N}$, $\Psi$ is defined on $B$ and $\Psi(B)\subset B$. To avoid technical difficulties which probably have little to do with the essence of the property of tethering, we replace ``every strongly convex regular geodesic ball'' with ``every ball of radius less than or equal to $r_{\textrm{cx}}$.'' Groisser's definition is more general than ours. Groisser conjectured that tethering ``might occur fairly generally.'' Several results in \cite{Groisser} can be strengthen if tethering assumption holds (even in this new sense). In the above theorem, we proved that for $t\in [0,1]$, the map $x\mapsto \exp_{x}(-t\nabla f_{2}(x))$ is tethered to $\{x_{i}\}_{i=1}^{N}$ in manifolds of constant nonnegative curvature. We conjecture that the same holds for manifolds of  nonnegative variable curvature. Based on the discussion in Remark \ref{rem:conv_neg_curv}, we conjecture that tethering in manifolds of negative curvature does not hold. As mentioned in Remark \ref{rem:conv_neg_curv} (and also expressed in Conjecture \ref{conj:optimal}), we conjecture that in order for $x\mapsto \exp_{x}(-t\nabla f_{2}(x))$ to map $B(o,\rho)\supset\{x_{i}\}_{i=1}^{N}$ to itself, $t$ should be smaller than $1$. More specifically, we conjecture that $t$ cannot be independent of $\rho$, and $t\in [0,\frac{1}{\textrm{c}_{\delta}(2\rho)}]$ suffices.
\end{remark}

\section{Convergence results for manifolds of arbitrary curvature}\label{sec:sub_optimal}
Here, we prove two classes of results which are sub-optimal compared to Conjecture \ref{conj:optimal}. In the first class the spread of data points is compromised to guarantee convergence. In the second class, the step-size is restricted more than what is needed to reduce the cost at each iteration to ensure that the iterates do not leave a neighborhood in which $\bar{x}_{2}$ the only zero of $\nabla f_{2}$.\footnote{The reader can convince herself or himself that Conjecture  \ref{conj:optimal} not only gives a smaller best a-priori $q$ than both Theorems \ref{theorem:compromised_L2} and \ref{theorem:compromise_step_size_L2} but also it gives a smaller asymptotic $q$ (see Remark \ref{rem:gradient_convexity}) when respective best a-priori step-sizes are employed. Notice that the asymptotic $q$ is data dependent.}
\subsection{Compromising the spread of data points}\label{sec:data_spread_compromize}
Here, by compromising the allowable spread (or support) of the data points we give a condition which guarantees convergence of Algorithm \ref{table:T1} to the global center of mass.

\begin{theorem}\label{theorem:compromised_L2}
Set $p=2$ and let $\bar{x}_{2}$ be the $L^{2}$ center of mass of $\{x_{i}\}_{i=1}^{N}\subset B(o,\rho)\subset M$ where $\rho\leq\frac{1}{3}r_{\textup{\textrm{cx}}}$. Define $t_{\delta,\rho}=\frac{1}{H_{B(o,3\rho)}}$, where $H_{B(o,3\rho)}=\textup{\textrm{c}}_{\delta}(4\rho)$ and $\textup{\textrm{c}}_{\kappa}$ is defined in (\ref{eq:c_kappa}).
In Algorithm \ref{table:T1} assume that $x^{0}\in B(o,\rho)$ and for every $k\geq 0$ choose $t_{k}=t$, where $t\in (0, 2t_{\delta,\rho})$. Then we have the following: The algorithm is well-defined for all $k\geq 0$ and each iterate of the algorithm continuously stays in $B(o,3\rho)$, $f_{2}(x^{k+1})\leq f_{2}(x^{k})$ for $k\geq 0$ (with equality only if $x^{k}$ is the Riemannian center of $\{x_{i}\}_{i=1}^N$), and $x^{k}\rightarrow \bar{x}_{2}$ as $k\rightarrow \infty$. Moreover, if $x^{0}$ coincides with $x^0$, then $\rho\leq \frac{1}{2}r_{\textup{\textrm{cx}}}$ is enough to guarantee the convergence, in which case each iterate of the algorithm continuously stays in $B(o,2\rho)$ and we can take $t_{\delta,\rho}=\frac{1}{H_{B(o,2\rho)}}$ where $H_{B(o,2\rho)}=\textup{\textrm{c}}_{\delta}(3\rho)$. More generally, for $2\leq p<\infty$ the same results hold if we replace $H_{B(o,3\rho)}$ and $H_{B(o,2\rho)}$, respectively, with $H_{B(o,3\rho),p}=(4\rho)^{p-2}\max\{p-1,\textup{\textrm{c}}_{\delta}(4\rho)\}$ and $H_{B(o,2\rho),p}=(3\rho)^{p-2}\max\{p-1,\textup{\textrm{c}}_{\delta}(3\rho)\}$.
\end{theorem}
\begin{proof}
For any $x\in M\setminus B(o,3\rho)$ we have $f_{2}(x)>2\rho^{2}>f_{2}(x^{0})$ (see (\ref{eq:fp})). From (\ref{eq:hessian_f2_bounds}) and (\ref{eq:c_kappa}) and that $\{x_{i}\}_{i=1}^{N}\subset B(o,\rho)$, one sees that $H_{B(o,3\rho)}=\textrm{c}_{\delta}(4\rho)$ is an upper bound on the eigenvalues of the Hessian of $f_{2}$ in $B(o,3\rho)$. Moreover, by Proposition \ref{prop:hessian_stepsize}, for small enough $t\in (0,2t_{\delta,\rho})$, $s\mapsto \exp_{x^{0}}(-s\nabla f_{2}(x^0))$ does not leave $B(o,3\rho)$ for $s\in [0,t]$, and we have $f(\exp_{x^{0}}(-t\nabla f_{2}(x^{0}))\leq f(x^{0})$, with equality only if $x^{0}$ is the unique zero of $\nabla f_{2}$ in $B(o,3\rho)$. However, $s\mapsto \exp_{x^{0}}(-s\nabla f_{2}(x^{0}))$ must lie in $B(o,\rho)$ for all $s$ in $(0,2t_{\delta,\rho})$, since on the boundary of $B(o,3\rho)$, $f$ is larger than $f(x^{0})$ and by continuity $s\mapsto \exp_{x^{0}}(-s\nabla f_{2}(x^{0}))$ cannot leave $B(o,3\rho)$ without making $f_{2}\big(\exp_{x^{0}}(-s\nabla f_{2}(x^{0}))\big)$ larger than $f_{2}(x^{0})$ inside $B(o,3\rho)$, which is a contradiction. Therefore, for any $t\in (0,2t_{\delta,\rho})$, the iterate $x^{1}=\exp_{x^{0}}(-t\nabla f_{2}(x^{0}))$ continuously stays in $B(o,3\rho)$ and $f_{2}(x^{1})\leq f_{2}(x^{0})$, with equality only if $x^{0}=\bar{x}_{2}$. A similar argument shows that for any $y\in B(o,3\rho)$ such that $f_{2}(y)\leq f_{2}(x^{0})$, $f(\exp_{y}(-s\nabla f_{2}(y)))$ for $s\in[0,t]$ belongs to $B(o,3\rho)$ and $f(\exp_{y}(-t\nabla f_{2}(y)))\leq f(y)$ with equality only if $y=\bar{x}_{2}$. In particular, assuming $x^{k}\in B(o,3\rho)$ and $f_{2}(x^{k})\leq f_{2}(x^{0})$, by setting $y=x^{k}\in B(o,3\rho)$, we conclude that $x^{k+1}$ continuously stays in $B(o,3\rho)$ and $f(x^{k+1})\leq f(x^{k})$ with equality only if $\bar{x}^{k}=\bar{x}_{2}$. Note that for any point $y$ in $B(o,3\rho)\setminus B(o,\rho)$ we have $d(y,x_i)<4\rho<\frac{2}{3}\textrm{inj}M$ for $1\leq i\leq n$; therefore, $\nabla f_{2}(y)$ in (\ref{eq:grad_Lp}) and hence (\ref{eq:gradient_descent}) in Algorithm \ref{table:T1} are well-defined. Next, by taking $B(o,3\rho)$ as $S$ in Theorem \ref{theorem:convergence_general}, we conclude that $x^{k}\rightarrow \bar{x}_{2}$ as $k\rightarrow \infty$. To see the claim about $B(o,2\rho)$, note that if $x^{0}$ coincides with $o$, then we have $f_{2}(x)>\frac{1}{2}\rho^{2}>f_{2}(x^{0})$ for any $x$ out of $B(o,2\rho)$ and the derived conclusions hold with $B(o,2\rho)$. The claims about $2\leq p<\infty$ follow similarly by further using (\ref{eq:hessian_fp_bounds}).
\end{proof}
The radius $3\rho$ above was found based on the simple observation that $f_{p}$ takes larger values outside of $B(o,3\rho)$ than inside of $B(o,\rho)$. Finer results should not be difficult to prove.

\subsection{Compromising the step-size}\label{sec:step_size_compromise}
Here, we briefly describe another approach in which the step-size is further restricted to ensure that the iterates do not leave a ball larger than $B(o,\rho)$- and not $B(o,\rho)$ itself. A rather similar idea has been used in \cite{Arnaudon1} and \cite{Yang1}, and here we partially follow the methodology in \cite{Yang1}. Specifically, given $\rho$ and $\rho'$ where $\rho<\rho'\leq r_{\textrm{cx}}$ and assuming the data points lie in $B(o,\rho)$, by restricting the step-size we want to make sure that, starting from $B(o,\rho')$, the iterates do not leave the larger ball $B(o,\rho')$.

For $x$ inside $B(o,\rho)$, let $t_{x}>0$ denote the first time $t\mapsto \gamma_{x}(t)=\exp_{x}(-t\nabla f_{2}(x))$ hits the boundary of $B(o,\rho')$. Note that $\sup_{B(o,\rho)}||\nabla f_{2}(x)||<2\rho$, therefore we must have $t_{x}>t_{x}^{\textrm{in}}$ where
\begin{equation}\label{eq:t_in}
t^{\textrm{in}}_{x}=\frac{\inf_{x\in B(o,\rho),y\in M\setminus B(o,\rho')} d(x,y)}{2\rho}=\frac{\rho'-\rho}{2\rho}.
\end{equation}
Similarly, for $y$ in the annular region between $B(o,\rho')$ and $B(o,\rho)$, let $t_{y}>0$ denote the first time $t\mapsto \gamma_{y}(t)=\exp_{y}(-t\nabla f_{2}(y))$ hits the boundary of $B(o,\rho')$. For $t\mapsto \frac{1}{2}d^{2}(o,\gamma_{y}(t))$ one writes the second order Taylor's series expansion in the interval $[0,t_{y}]$ as:
\begin{equation}
\frac{1}{2}d^{2}(o,\gamma_{y}(t_{y}))=\frac{1}{2}\rho'^{2}=\frac{1}{2}d^{2}(o,y)+\langle -\nabla f_{2}(y),-\exp_{y}^{-1}o\rangle t_{y}+\frac{1}{2}
\frac{\mathrm{d}^{2}f_{2,o}(t)}{\mathrm{d}t^{2}}\big|_{t=s}t_{y}^{2},
\end{equation}
where $s$ is in the interval $(0,t_{y})$. Next, using (\ref{eq:hessian_f2_bounds}) and noting that $\rho^2-d^{2}(o,y)>0$ we verify that
\begin{equation}
t_{y}> \frac{2\langle -\nabla f_{2}(y),\exp_{y}^{-1}o\rangle}{\textrm{c}_{\delta}(\rho')},
\end{equation}
where $\textrm{c}_{\delta}$ is defined in (\ref{eq:c_kappa}). Denote by $\angle x_{i}yo$ the angle, at $y$, between the minimal geodesics from $y$ to $x_{i}$ and from $y$ to $o$. It is shown in Lemma 10 in \cite{Yang1} that
\begin{equation}\label{eq:yang}
\cos \angle x_{i}yo \geq \frac{\textrm{sn}_{\Delta}(d(y,o)-\rho)}{\textrm{sn}_{\Delta}(d(y,o)+\rho)},
\end{equation}
where $\textrm{sn}_{\Delta}$ is defined in (\ref{eq:cot_hessian}). Using this and observing that $\|\nabla f_{2}(y)\|\geq d(y,o)-\rho$ we have $t_{y}> t_{y}^{\textrm{out},1}$, where
\begin{equation}\label{eq:tout_1}
t_{y}^{\textrm{out},1}= \frac{2}{\textrm{c}_{\delta}(\rho')}\times d(y,o)\times \big(d(y,o)-\rho\big)\times\frac{\textrm{sn}_{\Delta}(d(y,o)-\rho)}{\textrm{sn}_{\Delta}(d(y,o)+\rho)}.
\end{equation}
Also observe that (trivially) we must have $t_{y}>t_{y}^{\textrm{out},2}$, where
\begin{equation}\label{eq:tout_2}
t_{y}^{\textrm{out},2}=\frac{\rho'-d(y,o)}{\rho+d(y,o)}.
\end{equation}
Obviously, $t_{y}$ must satisfy $t_{y}> \max\{t_{y}^{\textrm{out},1},t_{y}^{\textrm{out},2}\}$. Define
\begin{equation}\label{eq:compromosied_step}
t_{\textrm{exit}}=\min\{t^{\textrm{in}},\inf_{y:\rho\leq d(y,o)<\rho'}\max\{t_{y}^{\textrm{out},1},t_{y}^{\textrm{out},2}\}\}
\end{equation}
where $t^{\textrm{in}}$, $t_{y}^{\textrm{out},1}$, and $t_{y}^{\textrm{out},2}$ are defined in (\ref{eq:t_in}), (\ref{eq:tout_1}), and (\ref{eq:tout_2}), respectively, with the assumption $\rho< \rho'\leq r_{\textrm{cx}}$.  We see that for any $z\in B(o,\rho')$ and any $t\in [0,t_{\textrm{exit}}]$, $\exp_{z}(-t\nabla f_{2}(z))$ belongs to $B(o,\rho')$. Notice that $t=t_{\textrm{exit}}$ is indeed acceptable. Also observe that $t_{\textrm{exit}}$ is larger than zero; since otherwise it can be zero only if for $z$ in the region $B(o,\rho')\setminus B(o,\rho)$ and very close to the boundaries of the region $t_{z}^{\textrm{out},1}$ and $t_{z}^{\textrm{out},2}$ both become arbitrary close to zero, which obviously cannot happen. Based on this analysis we have the following theorem.

\begin{theorem}\label{theorem:compromise_step_size_L2}
Let $p=2$, $\{x_{i}\}_{i=1}^{N}\subset B(o,\rho)$ and assume $\rho<\rho'\leq r_{\textup{\textrm{cx}}}$. Define
$H_{B(o,\rho')}=\textup{\textrm{c}}_{\delta}(\rho'+\rho)$ and set
\begin{equation}\label{eq:compromised_step_size_with_hessian}
t_{\delta,\Delta,\rho,\rho'}^{*}=\min\{t_\textup{\textrm{exit}},\frac{1}{H_{B(o,\rho')}}\},
\end{equation}
where $t_\textup{\textrm{exit}}$ is defined in (\ref{eq:compromosied_step}). In Algorithm \ref{table:T1}, choose an initial condition $x^{0}\in B(o,\rho)$ \footnote{In fact, according to the derivations, one could choose $x^{0}\in B(o,\rho')$.} and step-size $t_{k}=t$, where $t\in (0,2t_{\delta,\Delta,\rho,\rho'}^{*})\cap [0,t_\textup{\textrm{exit}}]$. Then we have the following: The algorithm is well-defined for every $k\geq 0$, each iterate continuously stays in  $B(o,\rho')$, $f_{2}(x^{k+1})\leq f_{2}(x^{k})$ with equality only if $x^{k}=\bar{x}_{2}$, and $x^{k}\rightarrow \bar{x}_{2}$ as $k\rightarrow \infty$.
\end{theorem}
\begin{proof}
The fact that each iterate continuously stays in $B(o,\rho')$ follows from preceding arguments. From this it we see that $d(x^{k},x_{i})<\textrm{inj}M$ for every $k\geq 0$ and $1\leq i\leq N$, and hence the algorithm is well-defined for $k\geq 0$. The rest of the claims follow from Theorem \ref{theorem:convergence_general}.
\end{proof}
Next, we give some numerical examples about the interplay between $\rho$, $\rho'$, and the step-sizes according to Theorem \ref{theorem:compromise_step_size_L2} and compare that with step-size and allowable spread from Conjecture \ref{conj:optimal} and Theorem \ref{theorem:compromised_L2}. First, let $\delta =0$ and $\Delta>0$ and let $\rho'=r_\textrm{cx}$. To have $t_{k}=1$ we need to have $\rho\leq r_{1}\approx  0.0303 r_{\textrm{cx}}$, while Theorem \ref{theorem:compromised_L2} gives much larger $\rho$, i.e., $\rho\leq \frac{1}{3}r_{\textrm{cx}}$. We can increase $\rho$ and further restrict the step-size: If we set $\rho=\frac{1}{3}r_{\textrm{cx}}$, then we get $t_{\delta,\Delta,\rho,\rho'}^{*}\approx 0.3965$, if $\rho=\frac{9}{10}r_{\textrm{cx}}$ we get $t_{\delta,\Delta,\rho,\rho'}^{*}=0.0353$, and finally when $\rho=0.99 r_{\textrm{cx}}$ we get $t_{\delta,\Delta,\rho,\rho'}^{*}=0.0033$, all of which are considerably smaller than the optimal step-size of $1$ in Conjecture \ref{conj:optimal}. Yet, the added value is that we have convergence for more spread-out data points (i.e., going from $\rho\leq \frac{1}{3}r_{\textrm{cx}}$ to \emph{almost} $\rho\leq r_{\textrm{cx}}$).  Next, let $\delta<0$, $\Delta=0$, and $\rho'=\frac{\pi}{2}\sqrt{-\delta}$ (this is just an arbitrary number). To get the optimal step-size in Theorem \ref{theorem:rate_of_convergence} which is $\frac{1}{H_{B(o,\rho')}}$ and is equal to $\frac{1}{\textrm{c}_{\delta}(\rho+\rho')}$, we need $\rho\leq r_{2}\approx 0.1950\rho'$. Therefore, the $t_{\delta,\Delta,\rho,\rho'}^{*}$ from Theorem \ref{theorem:compromise_step_size_L2} cannot be larger than the $t_{\delta,\rho}$ from Theorem \ref{theorem:compromised_L2}. In fact, if set $\rho=\frac{1}{3}\rho'$, then we need $t_{\delta,\Delta,\rho,\rho'}^{*}=0.3022$ according to Theorem \ref{theorem:compromise_step_size_L2}, while we have $t_{\delta,\rho}=0.4632$ from Theorem \ref{theorem:compromised_L2}.

Finer analysis could yield a larger estimate for the exit time than (\ref{eq:compromosied_step}). However, since $\textrm{c}_{\delta}(2\rho)$ is a upper bound on the eigenvalues of the Hessian of $f_{2}$ in $B(o,\rho)$, such an improvement will not result in an optimal step-size better than $t_{k}=\textrm{ct}_{\delta}(2\rho)^{-1}$ (cf. (\ref{eq:compromised_step_size_with_hessian}) and Conjecture \ref{conj:optimal}).

\section{On the configuration of data points and the local rate of convergence}\label{sec:RateOfConvergence}
In this section, we give a qualitative answer to the following question:``For which configurations of data points Algorithm \ref{table:T1} locates the center of mass very fast? very slowly?'' We use the facts mentioned in Subsection \ref{sec:speed_of_convergence} to answer this question.

We assume $p=2$. From Theorem \ref{theorem:rate_of_convergence} and the definition of $q$ in (\ref{eq:c_q}) it is clear that in addition to $\alpha$ the ratio $\frac{h_{S}}{H_{S}}$ is also important in determining the speed of convergence, and the asymptotic speed of convergence depends on the ratio $\frac{h_{S}}{H_{S}}$ in a very small neighborhood $S$ around $\bar{x}$.  Obviously, the smaller the ratio is, the slower the convergence will be, and vice versa. In the Euclidean case the ratio is $1$ and with $\alpha=1$ we have $q=0$; therefore, Algorithm \ref{table:T1} finds the center of mass in one step (see (\ref{eq:rete_of_conv}) and (\ref{eq:c_q})). However, in a curved manifold the ratio $\frac{h_{S}}{H_{S}}$ can be very small, due to drastic difference in the behavior of the Hessian of the distance function along different directions. Next we give simple examples that demonstrate this fact.

We consider the case of constant curvature since in this case the eigenvalues of the Hessian of the distance function are the same along all directions but the radial direction. Furthermore, let us assume $M$ is a $2$-dimensional simply connected manifolds with constant curvature, that is $M=\mathbb{S}^{2}_{\Delta}$ where $\Delta=1$ or $\Delta=-1$ (with the convention $\mathbb{S}^{2}_{1}\equiv \mathbb{S}^{2}$). We construct two simple configurations for which Algorithm \ref{table:T1} converges very fast and very slowly, respectively. Consider four data points $\{x_{i}\}_{i=1}^{4}$ and the closed ball $\bar{B}(o,\rho)\subset M$, where $\rho<r_{\textrm{cx}}$. Assume $x_{1}$ and $x_{2}$ are on the boundary of the ball in antipodal positions and that $x_{3}$ and $x_{4}$ are also in antipodal positions such that the geodesic $\gamma_{ox_{1}}$ from $o$ to $x_{1}$ and the geodesic $\gamma_{o,x_{3}}$ from $o$ to $x_{3}$ are perpendicular at $o$. We denote this configuration by $\bullet\displaystyle_{\bullet}^{\bullet}\bullet$. Obviously, $\bar{x}=o$ is the center of mass of $\{x_{i}\}_{i=1}^{4}$ with equal weights. It is easy to verify that for the $\bullet\displaystyle_{\bullet}^{\bullet}\bullet$ configuration the Hessian of $f_{2}$ at $x=o$ both along $\gamma_{ox_{1}}$ and along $\gamma_{ox_{3}}$ has eigenvalue $\frac{1}{2}(\rho\textrm{ct}_{\Delta}(\rho)+1)$. Consequently, at $o$ the ratio of the smallest and largest eigenvalue is $1$, hence $\frac{h_{S}}{H_{S}}\approx1$ around $\bar{x}=o$; and therefore, one expects that the local rate of convergence will be very fast. The opposite configuration is $\displaystyle_{\bullet}^{\bullet}$, that is, when $x_{3}$ and $x_{4}$ coincide with $x_{1}$ and $x_{2}$, respectively. In this case, at $x=o$ along $\gamma_{ox_{1}}$ the Hessian of $f_{2}$ has eigenvalue of $1$ and in the perpendicular direction it has eigenvalue $\rho\textrm{ct}_{\Delta}(\rho)$. Therefore, if $\Delta=1$, we have $\frac{h_{S}}{H_{S}}\approx \rho\cot\rho$ around $o$ which, in particular, can be very small if $\rho$ is close to $\frac{\pi}{2}$. If $\Delta=-1$, we have $\frac{h_{S}}{H_{S}}\approx(\rho\coth\rho)^{-1}$, which again can be small if $\rho$ is large. It is well known that the shape of the level sets of a function in a neighborhood of a minimizer is related to the ratio $\frac{h_{S}}{H_{S}}$. 
If the level sets are very elongated or thin, this means that the Hessian has very small eigenvalues along longitudinal directions and very large eigenvalues along the lateral directions and hence $\frac{h_{S}}{H_{S}}$ can be very small. For our two configurations, we encourage the reader to compare the shapes of the level sets of $f_{2}$ in $B(o,\rho)\subset \mathbb{S}^{2}$ for levels close to $f_{2}(o)$ (especially when $\rho$ is close to $\frac{\pi}{2}$) with the level sets of $f_{2}$ in $B(o,\rho)\subset \mathbb{S}_{-1}^{2}$ for levels close to $f_{2}(o)$ when $\rho$ is very large).

As a tangible example, on the standard unit sphere $\mathbb{S}^{2}$ we run Algorithm \ref{table:T1} for both the configurations with two different values of $\rho\approx 0.35\pi$ and $\rho\approx 0.47\pi$. The initial condition is chosen \emph{randomly}. The step-size is chosen as $t_{k}=1$. Figure \ref{fig:Example1} shows the distance $d(x^{k},\bar{x})$ in terms of the iteration index $k$. It is clear that for the $\displaystyle_{\bullet}^{\bullet}$ configuration the convergence is slower than the convergence for the $\bullet\displaystyle_{\bullet}^{\bullet}\bullet$ configuration, and as $\rho$ increases, convergence for both configurations becomes slower. However, for the $\displaystyle_{\bullet}^{\bullet}$ configuration as $\rho$ approaches $\frac{\pi}{2}$, the convergence becomes extremely slow and the $\bullet\displaystyle_{\bullet}^{\bullet}\bullet$ configuration is much more robust in that sense. Note that when $\rho\approx \frac{\pi}{2}$ even the center of mass of the $\displaystyle_{\bullet}^{\bullet}$ configuration is on the verge of non-uniqueness and this causes further (error) sensitivity and hence poor convergence (see \cite{BijanPhd} on the issue of high noise-sensitivity of the Riemannian mean in positively curved manifolds).

Although our example is rare in statistical applications, in a more general setting also one expects that if the configuration of data points is such that the convex hull of the data points has an elongated shape (especially if the length of the convex hull is large), then locating the Riemannian center of mass becomes a difficult problem (with the exception of the Euclidean case). Our analysis  does not tell the whole story in the case of non-constant curvature and we need more detailed analysis that takes into account the variability of eigenvalues of the Hessian of the distance function along non-radial directions, as well.

\begin{figure}[hpt]
  \begin{center}
  \scalebox{.65}{\includegraphics{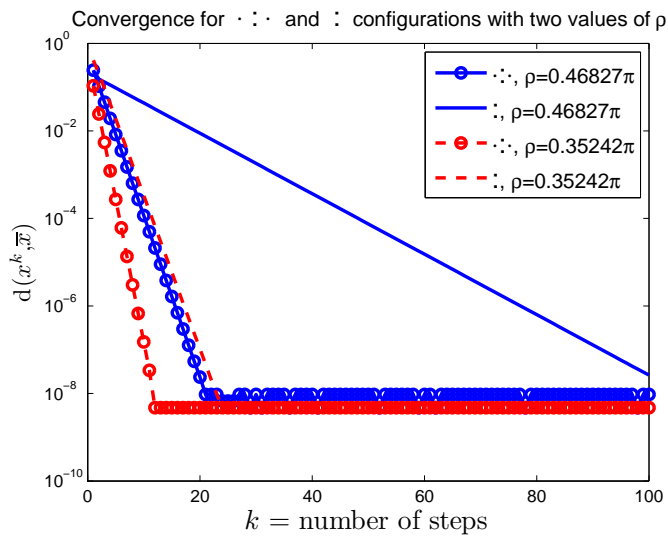}}
  \caption[]{Convergence behavior of Algorithm \ref{table:T1} with step-size $t_{k}=1$ for locating the center of mass two data point configurations denoted by $\bullet\displaystyle_{\bullet}^{\bullet}\bullet$ and $\displaystyle_{\bullet}^{\bullet}$ on the unit sphere $\mathbb{S}^{2}$. 
  }\label{fig:Example1}
\end{center}
\end{figure}

\section{Concluding Remarks}\label{sec:conc}
Our goal has been to give and prove the the best possible conditions for convergence of the popular constant step-size gradient descent for finding the Riemannian center of mass. We argued that Conjecture \ref{conj:optimal} gives such best conditions (in some specific yet general sense). The proof of the conjecture seems to be difficult, because, in particular, it appears that such a proof requires some very deep understanding about the behavior of the exponential map of a Riemannian manifold. We proved the conjecture for manifolds of constant nonnegative curvature and our proof was based on comparison Theorem \ref{theorem:comparison_sphere} which seems to give a better estimate of the behavior of the exponential map than what (possibly) could be gained by standard comparison theorems. Therefore, extending this comparison theorem (in appropriate sense) to manifolds of variable curvature not only could help prove Conjecture \ref{conj:optimal}, but also could provide deeper understanding of the behavior of the exponential map of a general manifold. Our Theorems \ref{theorem:compromised_L2} and \ref{theorem:compromise_step_size_L2} (whose proofs are based on simple observations) give weaker convergence conditions, but still these results are considerably better than the available ones. Finer analysis could improve the results of these theorems, as well.
\bibliography{OnTheConvergenceOfGradientDescentForFindingTheRiemannianCenterOfMass}
\bibliographystyle{plain}

\end{document}